\documentclass[12pt]{article}
\usepackage{amssymb,amsmath}
\usepackage{amsbsy}
\usepackage{amsthm}
\usepackage{multirow}
\usepackage{mathrsfs}
\usepackage{verbatim}
\usepackage[colorlinks]{hyperref}
\usepackage{fullpage}
\usepackage{mathdots}
\usepackage{graphicx,subfigure}
\usepackage[english]{babel}
\usepackage[utf8]{inputenc}
\usepackage{tikz}
\usepackage{mathtools}

\DeclarePairedDelimiter\floor{\lfloor}{\rfloor}

\usepackage{authblk}

\newtheorem{Theorem}[equation]{Theorem}
\newtheorem{Corollary}[equation]{Corollary}
\newtheorem{Lemma}[equation]{Lemma}
\newtheorem{Proposition}[equation]{Proposition}

\theoremstyle{definition}
\newtheorem{Definition}[equation]{Definition}
\newtheorem{Example}[equation]{Example}

\newtheorem{Convention}[equation]{Convention}

\theoremstyle{remark}

\newtheorem{Remark}[equation]{Remark}

\numberwithin{equation}{section}
\numberwithin{figure}{section}

\newcommand{\PP}{{\mathbb P}}
\newcommand{\C}{{\mathbb C}}

\newcommand{\R}{{\mathbb R}}

\newcommand{\N}{{\mathbb N}}

\newcommand{\mc}[1]{\mathcal{#1}}

\begin{document}

\title{Counting Borel Orbits in Symmetric Varieties of Types $BI$ and $CII$}

\author[1]{Mahir Bilen Can}
\author[2]{\"Ozlem U\u{g}urlu}

\affil[1]{{\small Tulane University, New Orleans; mahirbilencan@gmail.com}}    
\affil[2]{{\small Tulane University, New Orleans; ougurlu@tulane.edu}}

\normalsize

\date{January 16, 2018}
\maketitle

\begin{abstract}
This is a continuation of our combinatorial 
program on the enumeration of Borel orbits in 
symmetric varieties of classical types. 
Here, we determine the generating 
series the numbers of Borel orbits in 
$\mathbf{SO}_{2n+1}/\mathbf{S(O}_{2p}\times \mathbf{O}_{2q+1}\mathbf{)}$ (type $BI$)
and in $\mathbf{Sp}_n/\mathbf{Sp}_p\times \mathbf{Sp}_q$ (type $CII$).
In addition, we explore relations to lattice path enumeration.

\vspace{.2cm}

\end{abstract}

\section{Introduction}\label{S:Introduction}

The purpose of our paper is to 
continue the program that is initiated 
in our previous paper \cite{CU}, which
is about finding generating functions and 
their combinatorial interpretations for certain families 
of involutions, called clans, in Weyl groups. 
There is an important motivation for 
undertaking such a  task and it comes 
from a desire to better understand the 
cohomology rings of homogeneous varieties of the form 
$G/K$, where $K$ is the fixed 
subgroup of an involutory automorphism of a complex
reductive group $G$. 
In other words, there exists an automorphism $\theta: G\rightarrow G$
such that $\theta^2=id$ and $K=\{g\in G:\ \theta(g)=g \}$. 
Such a coset space is called a symmetric variety.

The study of symmetric varieties form an integral part of geometry
and many interesting manifolds are (locally) 
diffeomorphic to a symmetric variety.
For example, $n-1$ dimensional sphere in $\R^n$ 
can be recognized as $G(\R)/K(\R)$, where 
$G$ is $\mathbf{SO}_n$, the special orthogonal group of linear transformations
with determinant 1, and $K$ is its subgroup $ \mathbf{S(O_{n-1}\times O_1)}$
consisting of block matrices of the form 
$\begin{pmatrix}
A & 0 \\
0 & \pm 1
\end{pmatrix}$ where $A$ is an orthogonal matrix of order $n-1$.

Among the important properties of a symmetric variety 
are the following:
\begin{enumerate}
\item[(i)] $K$ is reductive, hence $G/K$ is affine 
as an algebraic variety.
\item[(ii)] The isotypic components of the 
coordinate ring $\C[G/K]$ are multiplicity-free as $G$-modules.
\end{enumerate}
The second listed property amounts to 
$G/K$ having finitely many orbits under the 
left translation action of a Borel subgroup of $G$. 
(See \cite{Brion}.) 
Here, by a Borel subgroup we mean a subgroup $B$ 
of $G$ which is maximal among connected
solvable subgroups. (For example, the subgroup
consisting of upper triangular matrices in 
the general linear group $\mathbf{GL}_n$ is 
a Borel subgroup.) 
Another good reason for studying Borel 
orbits in $G/K$ comes from the 
fact that the $B$-orbits in $G/K$ are in 
1-1 correspondence with the $K$-orbits in 
$G/B$ and the topology of the latter 
(flag) variety is completely determined by
the inclusion order on the 
set of Borel orbit closures. 
Therefore, it is a rather natural and important 
question to determine the number of Borel
orbits in $G/K$.

By a {\em classical symmetric variety} 
we mean one of the following homogeneous varieties 
\begin{table}[h!]
\centering
 \begin{tabular}{ c || c } 
 Type & Symmetric variety \\
\hline
\textbf{A0} & $\mathbf{SL}_n\times \mathbf{SL}_n/\mathbf{diag(SL}_n\mathbf{)}$ \\ 
\hline
\textbf{AI} & $\mathbf{SL}_n/\mathbf{SO}_n$ \\ 
\hline
\textbf{AII} & $\mathbf{SL}_{2n}/\mathbf{Sp}_n$ \\ 
\hline
\textbf{AIII} & $\mathbf{SL}_n/\mathbf{S(GL}_p\times \mathbf{GL}_q\mathbf{)}$, where $p+q=n$ \\ 
\hline
\textbf{B0} & $\mathbf{SO}_n\times \mathbf{SO}_n/\mathbf{diag(SO}_n\mathbf{)}$ \\ 
\hline
\textbf{BI} & $\mathbf{SO}_{2n+1}/\mathbf{S(O}_{2p}\times \mathbf{O}_{2q+1}\mathbf{)}$, where $p+q=n$ \\ 
\hline
\textbf{C0} & $\mathbf{Sp}_n\times \mathbf{Sp}_n/\mathbf{diag(Sp}_n\mathbf{)}$ \\ 
\hline
\textbf{CI} & $\mathbf{Sp}_n/\mathbf{GL}_n$ \\ 
\hline
\textbf{CII} & $\mathbf{Sp}_n/\mathbf{Sp}_p\times \mathbf{Sp}_q$, where $p+q=n$ \\ 
\hline
\textbf{DI} & $\mathbf{SO}_{2n}/\mathbf{S(O}_{p'}\times \mathbf{O}_{q'}\mathbf{)}$, where $p'+q'=2n$ \\ 
\hline
\textbf{DIII} & $\mathbf{SO}_{2n}/\mathbf{GL}_n$ \\ 
 \end{tabular}
 \caption{Classical symmetric varieties}
 \label{F:Table1}
\end{table} 
In this manuscript, somewhat loosely following
our previous work, where we studied the combinatorics
of the generating functions regarding the number of $B$-orbits 
in type $\textbf{AIII}$, we will give a count of the 
$B$-orbits for the cases of $\mathbf{BI}$ and $\mathbf{CII}$.
These two types correspond to decompositions of the 
vector spaces $\C^{2n+1}$ and $\C^{2n}$, respectively,  
into two orthogonally complementary subspaces with respect 
to a symmetric and a skew-symmetric bilinear
form. See~\cite{Howe}.

By looking at the orbits of the two sided 
action of $K\times B$ on $G$, it is easy to see 
that the cardinality of the set of $B$-orbits 
in $G/K$ is the same as the cardinality of the set of 
$K$-orbits in $G/B$. We know from \cite{WyserII} that,
for a symmetric variety $G/K$ as in Table~\ref{F:Table1},
the combinatorial objects parameterizing $K$-orbits
in $G/B$ have a rather concrete description; they 
are called ``clans'' with suitable adjectives. 
This nomenclature has first appeared in 
a paper of Matsuki and Oshima~\cite{MO}.
In~\cite{Yamamoto}, Yamamoto used these objects to 
determine the image of the moment map 
of the conormal bundle of $K$ orbits in $G/B$. 
(She worked with types \textbf{AIII} and \textbf{CII} only.) 
As far as we are aware of, 
after Yamamoto's work on clans, there was a long 
pause on the study of these combinatorial objects
until McGovern's work in~\cite{Monty09}
and Wyser's 2012 thesis~\cite{WyserI}, where 
Wyser clarified many obscurities around the definition of clans.
In~\cite{WyserII} he used them to study degeneracy 
loci and in~\cite{WyserIII} he used them to 
study the weak and strong Bruhat orders on $K$ orbit closures
in $G/B$ (in type \textbf{AIII}).
More recent work on the combinatorics of type \textbf{AIII} 
clans and applications to geometry can be found in~\cite{WyserIV}.
See~\cite{CJW} also.

We will refer to the clans corresponding to the Borel 
orbits in $\mathbf{Sp}_n/\mathbf{Sp}_p\times \mathbf{Sp}_q$ as ssymmetric $(2p,2q)$ clans
and we will call the clans corresponding to the 
Borel orbits of $\mathbf{SO}_{2n+1}/\mathbf{S(O}_{2p}\times \mathbf{O}_{2q+1}\mathbf{)}$
the symmetric $(2p,2q+1)$ clans.
However, we should mention that these names are local to our paper. 
The definitions of symmetric and 
ssymmetric clans are rather lengthy, so, we postpone 
their precise definitions to the preliminaries 
section and introduce the notation for their collections 
and the corresponding cardinalities only. 
\begin{align*}
BI(p,q) &:= \{ \text{symmetric $(2p,2q+1)$ clans} \}, \hspace{.5cm} b_{p,q} := \# BI(p,q);\\
CII(p,q) &:= \{ \text{ssymmetric $(2p,2q)$ clans} \},\hspace{.5cm} c_{p,q} := \# CII(p,q). 
\end{align*}

As we mentioned before, 
clans are in bijection with ``signed'' involutions. 
A signed involution is an involution in $S_n$, for some $n$,
such that each fixed point of the involution is labeled with a 
$+$ sign or with a $-$ sign. 
Assuming the existence of a particular such bijection, 
which we will present in the sequel, 
we proceed to 
denote by $\beta_{k,p,q}$ the 
number of symmetric $(2p,2q+1)$ clans whose
corresponding involution has exactly $k$ 2-cycles 
as a permutation. In a similar way, we denote 
by $\gamma_{k,p,q}$ the number of ssymetric $(p,q)$ clans 
whose corresponding involution has $k$ 2-cycles. 
Clearly, 
$$
b_{p,q} = \sum_k \beta_{k,p,q} \ \text{ and } \ 
c_{p,q} = \sum_k \gamma_{k,p,q}. 
$$
Our goal in this manuscript is to present various 
formulas and combinatorial interpretations for $\beta_{k,p,q}$'s,
$\gamma_{k,p,q}$'s, and foremost, for $b_{p,q}$'s and $c_{p,q}$'s.  

\begin{Convention}\label{A:Convention}
If $p$ and $q$ are two nonnegative integers such that $q \geq p$, then 
we assume that $\beta_{k,p,q}=0$ for all $0\leq k \leq 2q+1$.
\end{Convention}

Now we are ready to describe our results in more detail.
First of all, by analyzing 
the structure of symmetric clans 
we prove the following result:
\begin{Theorem}\label{T:first}
Let $p$ and $q$ be two nonnegative integers 
such that $p>q$. Then for every nonnegative 
integer $k$ with $k\leq 2q+1$, we have
\begin{align}\label{OS:first formula1}
\beta_{k,p,q} =
\begin{cases}
 {{n-2l} \choose {p-l} }{n \choose 2l} a_{2l} & \text{ if } \; k=2l; \\
  {{n-(2l+1)} \choose {p-(l+1)}} {n \choose 2l+1} a_{2l+1}& \text{ if } \; k=2l+1,
\end{cases}
\end{align}
where 
\begin{align}\label{A:formulafora}
a_{2l} := \sum_{b=0}^{l} {2l \choose 2b} \frac{(2b)!}{b!}
\hspace{.75cm} \text{  and }  \hspace{.75cm}
a_{2l+1} := \sum_{b=0}^{l} {{2l+1} \choose 2b} \frac{(2b)!}{b!}.
\end{align}
In particular we have 
$$
b_{p,q}  = \sum_{l=0}^q \left( {{n-2l} \choose {p-l} }{n \choose 2l} a_{2l}
+   {{n-(2l+1)} \choose {p-(l+1)}} {n \choose 2l+1} a_{2l+1}\right).
$$
\end{Theorem}

The following formulae for the number of Borel orbits in 
$\mathbf{SO}_{2n+1}/\mathbf{S(O}_{2p}\times \mathbf{O}_{2q+1}\mathbf{)}$
for $q=0,1,2$ is now a simple consequence of our Theorem~\ref{T:first}.
\begin{align*}
b_{p,0} &= p+1 \\
b_{p,1} &= (p+1)a_0 + p(p+1)a_1 + \frac{p(p+1)}{2}a_2 + \frac{p(p+1)(p-1)}{6}a_3\\
&= \frac{7p^3 +15p^2+ 14p+6}{6}\\
b_{p,2} &= {p+2 \choose 2}\biggr( \frac{81p^3 +22 p^2 +137 p +60}{60} \biggr) \\
&= \frac{81p^5 +265 p^4 +365 p^3 +515 p^2 +454 p +120}{120}
\end{align*}

Theorem~\ref{T:first} tells us that, for every fixed $q$, the integer  
$b_{p,q}$ can be viewed as a specific value of 
a polynomial function of $p$. 
However, it is already apparent from the case of $q=1$
that this polynomial may have non-integer coefficients.
We conjecture that $q=0$ is the only case where $p\mapsto b_{p,q}$
is a polynomial function with integral coefficients. 
We conjecture also that for every nonnegative integer $q$, 
as a polynomial in $p$, $b_{p,q}$ is unimodal.

Note that the numbers $a_{2l}$ and $a_{2l+1}$ in Theorem~\ref{T:first}
($l=0,1,\dots,q$) are special 
values of certain hypergeometric functions. More precisely, 
\begin{align*}
a_{2l}  = \biggr(\frac{-1}{4}\biggr)^{-l} U\biggr(-l, \frac{1}{2}, -\frac{1}{4}\biggr), \\
a_{2l+1} = \biggr(\frac{-1}{4}\biggr)^{-l} U\biggr(-l, \frac{3}{2}, -\frac{1}{4}\biggr).
\end{align*}
where $U(a,b,z)$ is the confluent hypergeometric function of the second kind. 
Such functions form one of the two distinct families of hypergeometric
functions which solves the Kummer's differential equation 
\begin{align}\label{A:Kummers}
z y'' + (c-z) y' - a y = 0
\end{align}  
for some constants $a$ and $c$.
Kummer's ODE has a regular singular point at the origin
and it has an irregular singularity at infinity.

The expressions in (\ref{OS:first formula1}) are
too complicated for practical purposes, therefore we seek for 
better expressions in the forms of recurrences 
and generating functions for $\beta_{k,p,q}$'s. 
It turns out that between various $\beta_{k,p,q}$'s there are four 
``easy-to-derive'' recurrence relations as in (\ref{A:simplerecs}), 
and there are four ``somewhat easy-to-derive'' recurrence relations 
as in (\ref{A:kevenfr}), (\ref{A:koddfr}), 
(\ref{A:kevensr}), and (\ref{A:koddsr}). 
(We are avoiding showing these recurrences on purpose 
since they, especially the latter four, are rather lengthy.)
The first four relations do not mix $k$'s and they are linear. 
The second four recurrences are 3-term nonlinear recurrence relations 
and they do not mix $p,q$'s. 
Moreover, the relations (\ref{A:kevenfr}) and (\ref{A:koddfr}) are 
intervowen in the sense that both of them use 
consecutive terms in $k$'s. The relations
(\ref{A:kevensr}) and (\ref{A:koddsr}) maintain the parity of $k$,
however, their coefficients are more complicated than 
the previous two.

It is not futile to expect that the eight recurrence relations 
we talked about 
lead to a manageable generating function, hence to  
a new formulation of the numbers $\beta_{k,p,q}$. 
We pursued this approach 
by studying the generating polynomial 
\begin{align}\label{A:Hpq}
h_{p,q}(x):=\sum_{k\geq 0} \beta_{k,p,q}x^k
\end{align}
and we run into some surprising complications.
After pushing our computations as much as possible 
by using the easier, interrelated relations (\ref{A:kevenfr}) and (\ref{A:koddfr}),
we arrived at a $4\times 4$ system of linear ODE's with an 
irregular singular point at the origin:
\begin{align}\label{A:mainODE}
x^3 X'
=\begin{bmatrix}
(2p+2q-1) x^2+ 2 & x & -4pqx & - (2q+1) \\
x & (2p+2q-1)x^2+2 & - 2p & -(4pq+2p-2q-1)x\\
x^3 & 0 & 0 & 0\\
0 & x^3 & 0 & 0 
\end{bmatrix}
 X, 
\end{align}
where 
\begin{align*}
X = \begin{bmatrix}
u(x)\\
v(x) \\
A_e(x)\\
A_o(x)
\end{bmatrix}\qquad \text{  with } \qquad
\begin{aligned}
 u(x):=&A_e^{'}(x), \\ 
 v(x):=&A_o^{'}(x), \\
 A_e(x) := &\sum_{l=0}^{q} \beta_{2l,p,q} x^{2l}, \\
 A_o(x) := &\sum_{l=0}^{q} \beta_{2l+1,p,q} x^{2l+1}.
 \end{aligned}
\end{align*}
Without worrying about convergence, we are able to 
formally solve this system of ODE's however our method
does not yield a satisfactorily clean formula. Instead,
it provides us with a sequence of computational steps which eventually
can be used for finding good approximations to $\beta_{k,p,q}$'s
for any $p,q$. In order for not to break 
the flow of our exposition we decided 
to postpone the explanation of the intricacies of 
(\ref{A:mainODE}) to the appendix. 
Let us mention in passing that this type 
of ODE's (that is linear ODE's with irregular singular points) gave 
impetus to the development of reduction theory
for connections where the structure group is an algebraic group. 
See~\cite{BV}.
Now in a sense we are working our way back to
such ODE's by trying to find
the refined numbers of Borel orbits in a symmetric variety.

\vspace{.5cm}

To break free from the difficulties caused by complicated 
interactions between $\beta_{k,p,q}$'s, we consider 
the following alternative to $h_{p,q}(x)$:
\begin{align*}
b_{p,q}(x) := \sum_{l=0}^{q} (\beta_{2l,p,q} x^{q-l} + \beta_{2l+1,p,q}x^{q-l} ).
\end{align*}
Clearly, $b_{p,q}(x)$ is a polynomial
of degree $q$ and similarly to $h_{p,q}(x)$ its evaluation at $x=1$ gives $b_{p,q}$.
Let $B_p(x,y)$ denote the following generating function, 
which is actually a polynomial due to our convention (\ref{A:Convention}):
\begin{align*}
 B_{p}(x,y) := \sum_{q \geq 0} b_{p,q}(x) y^q.
\end{align*}
Now, by using the previously mentioned recurrences, we observe that 
$$
b_{p,q}(x)' = (p+q) b_{p,q-1}(x).
$$
From here it is not difficult to 
write down the governing partial differential equation for $B_p(x,y)$;
\begin{align}\label{A:fundamentalPDE}
\frac{\partial}{\partial x} B_p(x,y) - y^2\frac{\partial}{\partial y} B_p(x,y)= y(1+p) B_p(x,y).
\end{align}
By solving (\ref{A:fundamentalPDE}) we record a generating polynomial identity.
\begin{Theorem}\label{T:second}
If $f_p(z)$ denotes the polynomial that is obtained from $B_p(1,y)$
by the transformation $y\leftrightarrow z/(1-z)$, then we have 
\begin{align}\label{A:prep1}
f_p(z)
&= (1+z)^{p+1} \left( (p+1) + 2 \sum_{q\geq 1} 
(\beta_{2q,p,q}+\beta_{2q+1,p,q})z^q \right),
\end{align}
where $a_k$'s are as in Theorem~\ref{T:first}.
\end{Theorem}

\vspace{.5cm}
Next, we proceed explain our results on the
number of Borel orbits in $\mathbf{Sp}_n/\mathbf{Sp}_p\times \mathbf{Sp}_q$.
Recall that the notation $\gamma_{k,p,q}$ stands for the number 
of ssymmetric $(2p,2q)$ clans whose corresponding involutions have 
exactly $k$ 2-cycles. 
By counting the number of possible choices 
for the 2-cycles and the fixed points in an involution corresponding to a 
ssymmetric $(2p,2q)$ clan, we obtain the following symmetric expression:
\begin{align}\label{A:gamma kqp1}
\gamma_{k,p,q} = \frac{(q+p)! }{ (q-k)! (p-k)! k!}.
\end{align}
Note that the formula in 
(\ref{A:gamma kqp1}) is defined independently of the inequality $q<p$,
therefore, $\gamma_{k,p,q}$'s are defined for all nonnegative integers 
$p,q$, and $k$ with the convention that $\gamma_{0,0,0}=1$. 
As we show in the sequel (Lemma~\ref{T:rec}) $\gamma_{k,p,q}$'s satisfy
a 3-term recurrence 
\begin{equation}\label{recurrence1}
\gamma_{k,p,q} =  \gamma_{k,p-1,q} +  
\gamma_{k,p,q-1} + 2(q+p-1) \gamma_{k-1,p-1,q-1}\; \; \text{and} \;\; 
\gamma_{0,p,q}= {p+q \choose p}.
\end{equation}
Consequently, we obtain our first result on the number of ssymmetric $(2p,2q)$ clans.
\begin{Proposition}\label{P:First recurrence}
If $p$ and $q$ are positive integers, then the number of ssymmetric $(2p,2q)$ clans
satisfies the following recurrence
\begin{equation}\label{A:recurrence1}
c_{p,q} =  c_{p-1,q} + c_{p,q-1} +  2(p+q-1) c_{p-1,q-1}.
\end{equation}
\end{Proposition}
At this point we start to notice some 
similarities between the combinatorics of ssymmetric clans 
and our work in \cite{CU}, where
we studied the generating functions and combinatorial interpretations
of the numbers of Borel orbits in symmetric 
varieties of type $\textbf{AIII}$. 
Following the notation from the cited reference, let us denote by $\alpha_{p,q}$
the number of Borel orbits in 
$\mathbf{SL}_n/\mathbf{S(GL}_p\times \mathbf{GL}_q\mathbf{)}$, where $p+q=n$.
Then we have  
\begin{align}\label{A:recurrence12}
\alpha_{p,q} = \alpha_{p-1,q} + \alpha_{p,q-1} + (p+q-1)\alpha_{p-1,q-1}
\end{align}
holds true for all $p,q\geq 1$. From this relation, we obtained many 
combinatorial results on $\alpha_{p,q}$'s in \cite{CU}. 
Here, by exploiting similarities between the two 
recurrences (\ref{A:recurrence1}) and (\ref{A:recurrence12}), 
we are able to follow the same route and obtain
analogues of all of the results of \cite{CU}. 
To avoid too much repetition, we will focus only
on the selected analogues of our results from the previous paper. 
First we have a result on the generating function for $c_{p,q}$'s.

Let $v(x,y)$ denote the bivariete generating function
\begin{align}\label{A:vxy}
v(x,y) = \sum_{p,q\geq 0} c_{p,q} \frac{(2x)^q y^p}{p!}.
\end{align}
As we show that in the sequel, $v(x,y)$ obeys a first
order linear partial differential equation of the form
\begin{align}\label{A:oursecondPDE}
(-2x^2) \frac{\partial v(x,y)}{\partial x} + (1-2x-4xy)\frac{\partial v(x,y)}{\partial y}=  (1+2x)\ v(x,y).
\end{align}
with initial conditions
$$
v(0,y) = e^y\ \text{ and } \ v(x,0) = \frac{1}{1-2x}.
$$
The solution of (\ref{A:oursecondPDE}) gives us a remarkable expression
for the generating function (\ref{A:vxy}) in suitably transformed coordinates.

\begin{Theorem}\label{T:main2}
Let $r$ and $s$ be two algebraically independent 
variables that are related to $x$ and $y$ by the relations 
\begin{align*}
x(r,s)= \frac{r}{2rs+1} \ \text{ and } \ 
y(r,s)=\frac{3s + 4r^2 s^3 - 6r^2s^2 + 6rs^2-6rs }{3(2rs+1)^2}.
\end{align*}
In this case, the generating function $v(x,y)$ 
of $c_{p,q}$'s in $r,s$-coordinates is given by 
\begin{align*}
v(r,s) = \frac{e^{s} (2rs+1)}{1-2r}.
\end{align*}
\end{Theorem}
\vspace{.5cm}

Next, we explain the most combinatorial results of our paper. 
The {\em $(p,q)$-th Delannoy number}, denoted by $D(p,q)$, 
is defined via the recurrence relation
\begin{align}\label{D:basic recurrence}
D(p,q)=D(p-1,q)+D(p,q-1)+D(p-1,q-1)
\end{align}
with respect to the initial conditions $D(p,0)=D(0,q)=D(0,0)=1$. 
It is due to the linear nature of (\ref{D:basic recurrence}) that 
the generating function for $D(p,q)$'s is relatively simple:
$$
\sum_{
\begin{subarray}{l}  
p+q \geq 0 \\ 
p,\, q\, \in \N
\end{subarray}}
D(p,q)x^{i}y^{j} = \frac{1}{1-x-y-xy}.
$$ 
One of the most appealing properties of Delannoy numbers
is that they give the count of lattice paths that move with
unit steps $E:=(1,0)$, $N:=(0,1)$, and $D:=(1,1)$ in the plane.
More precisely, $D(p,q)$ gives the number of lattice paths
that start at the origin $(0,0)\in \N^2$ and ends at $(p,q)\in \N^2$
moving with $E,N$, and $D$ steps only. 
We will refer to such paths as the Delannoy paths
and denote the set of them by $\mc{D}(p,q)$.
For example, if $(p,q)=(2,2)$, then $D(2,2)=13$.
In Figure~\ref{F:22} we listed the corresponding Delannoy paths.

\begin{figure}[h]
\begin{center}
\begin{tikzpicture}[scale=.5]

\begin{scope}[xshift=-15cm]
\draw (0,0) grid (2,2);
\draw[ultra thick,red] (0,0) to (2,0);
\draw[ultra thick,red] (2,0) to (2,2);
\end{scope}

\begin{scope}[xshift=-12.5cm]
\draw (0,0) grid (2,2);
\draw[ultra thick,red] (0,0) to (1,0);
\draw[ultra thick,red] (1,0) to (2,1);
\draw[ultra thick,red] (2,1) to (2,2);
\end{scope}

\begin{scope}[xshift=-10cm]
\draw (0,0) grid (2,2);
\draw[ultra thick,red] (0,0) to (1,0);
\draw[ultra thick,red] (1,0) to (1,1);
\draw[ultra thick,red] (1,1) to (2,1);
\draw[ultra thick,red] (2,1) to (2,2);
\end{scope}

\begin{scope}[xshift=-7.5cm]
\draw (0,0) grid (2,2);
\draw[ultra thick,red] (0,0) to (1,0);
\draw[ultra thick,red] (1,0) to (1,1);
\draw[ultra thick,red] (1,1) to (2,2);
\end{scope}

\begin{scope}[xshift=-5cm]
\draw (0,0) grid (2,2);
\draw[ultra thick,red] (0,0) to (1,0);
\draw[ultra thick,red] (1,0) to (1,1);
\draw[ultra thick,red] (1,1) to (1,2);
\draw[ultra thick,red] (1,2) to (2,2);
\end{scope}

\begin{scope}[xshift=-2.5cm]
\draw (0,0) grid (2,2);
\draw[ultra thick,red] (0,0) to (1,1);
\draw[ultra thick,red] (1,1) to (2,1);
\draw[ultra thick,red] (2,1) to (2,2);
\end{scope}

\begin{scope}[xshift=0cm]
\draw (0,0) grid (2,2);
\draw[ultra thick,red] (0,0) to (2,2);
\end{scope}

\begin{scope}[xshift=2.5cm]
\draw (0,0) grid (2,2);
\draw[ultra thick,red] (0,0) to (1,1);
\draw[ultra thick,red] (1,1) to (1,2);
\draw[ultra thick,red] (1,2) to (2,2);
\end{scope}

\begin{scope}[xshift=5cm]
\draw (0,0) grid (2,2);
\draw[ultra thick,red] (0,0) to (0,1);
\draw[ultra thick,red] (0,1) to (2,1);
\draw[ultra thick,red] (2,1) to (2,2);
\end{scope}

\begin{scope}[xshift=7.5cm]
\draw (0,0) grid (2,2);
\draw[ultra thick,red] (0,0) to (0,1);
\draw[ultra thick,red] (0,1) to (1,1);
\draw[ultra thick,red] (1,1) to (2,2);
\end{scope}

\begin{scope}[xshift=10cm]
\draw (0,0) grid (2,2);
\draw[ultra thick,red] (0,0) to (0,1);
\draw[ultra thick,red] (0,1) to (1,1);
\draw[ultra thick,red] (1,1) to (1,2);
\draw[ultra thick,red] (1,2) to (2,2);
\end{scope}

\begin{scope}[xshift=12.5cm]
\draw (0,0) grid (2,2);
\draw[ultra thick,red] (0,0) to (0,1);
\draw[ultra thick,red] (0,1) to (1,2);
\draw[ultra thick,red] (1,2) to (2,2);
\end{scope}

\begin{scope}[xshift=15cm]
\draw (0,0) grid (2,2);
\draw[ultra thick,red] (0,0) to (0,2);
\draw[ultra thick,red] (0,2) to (2,2);
\end{scope}

\end{tikzpicture}

\end{center}
\caption{Delannoy paths}
\label{F:22}
\end{figure}
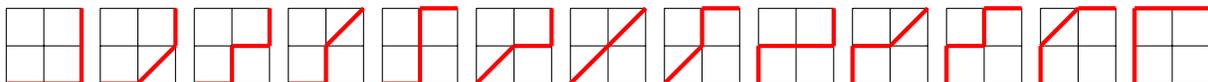

Let $L$ be a Delannoy path that ends at the 
lattice point $(p,q)\in \N$. 
We agree to represent $L$ as a word 
$L_1L_2\dots L_k$, where each $L_i$ $(i=1,\dots, k)$
is a pair of lattice points, say $L_i=((a,b),(c,d))$,
and $(c-a,d-b)\in \{N,E,D\}$. 
In this notation, we define the weight of the 
$i$-th step as 
\begin{align*}
weight(L_i) = 
\begin{cases}
1 & \text{ if } L_i =((a,b),(a+1,b)); \\
1 & \text{ if } L_i =((a,b),(a,b+1)); \\
2(a+b+1) & \text{ if } L_i= ((a,b),(a+1,b+1)).  
\end{cases}
\end{align*}
Finally, we define the weight of $L$, denoted by $\omega(L)$ 
as the product of the weights of its steps:
\begin{align}\label{A:weightofL}
\omega(L)  = weight(L_1) weight(L_2)\cdots weight(L_k).
\end{align}

\begin{Example}
Let $L$ denote the Delannoy path that is depicted in Figure~\ref{F:2D}.
In this case, the weight of $L$ is $\omega(L) = 6\cdot 12\cdot 16 = 1152$.

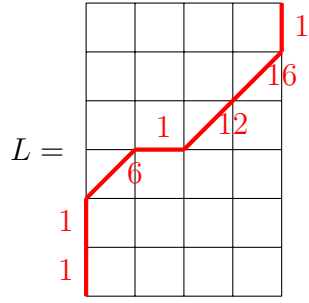
\begin{figure}[h]
\begin{center}
\begin{tikzpicture}[scale=.65]

\begin{scope}
\node at (-1,3) {$L=$};
\draw (0,0) grid (4,6);
\draw[ultra thick,red] (0,0) to (0,1) node[anchor=north east] {1};
\draw[ultra thick,red] (0,1) to (0,2) node[anchor=north east] {1};
\draw[ultra thick,red] (0,2) to (1,3) node[anchor=north] {6};
\draw[ultra thick,red] (1,3) to (2,3) node[anchor=south east] {1};
\draw[ultra thick,red] (2,3) to (3,4) node[anchor=north] {12};
\draw[ultra thick,red] (3,4) to (4,5) node[anchor=north] {16};
\draw[ultra thick,red] (4,5) to (4,6) node[anchor=north west] {1};
\end{scope}
\end{tikzpicture}
\end{center}
\caption{Two weighted Delannoy paths.}
\label{F:2D}
\end{figure}

\end{Example}

\begin{Proposition}\label{P:anotherpaththeorem}
Let $p$ and $q$ be two nonnegative integers and 
let $\mc{D}(p,q)$ denote the corresponding set of Delannoy 
paths. 
In this case, we have 
$$
c_{p,q} = \sum_{L\in \mc{D}(p,q)} \omega(L).
$$
\end{Proposition}

Although Proposition~\ref{P:anotherpaththeorem}
expresses $c_{p,q}$ as a combinatorial summation 
it does not give a combinatorial set of objects whose
cardinality is given by $c_{p,q}$. The last result our paper
offers such an interpretation. 
\begin{Definition}
A $k$-diagonal step (in $\N^2$) is a diagonal step $L$ of the form 
$L=((a,b),(a+1,b+1))$, where $a,b\in \N$ and $k=a+b+1$.
\end{Definition}
As an example, in Figure~\ref{F:4steps} we depict all $4$-diagonal steps in $\N^2$. 
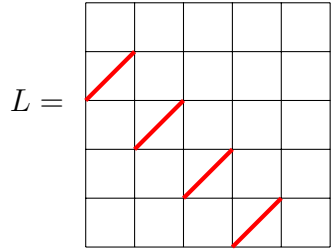
\begin{figure}[h]
\begin{center}
\begin{tikzpicture}[scale=.65]
\begin{scope}
\node at (-1,3) {$L=$};
\draw (0,0) grid (5,5);
\draw[ultra thick,red] (3,0) to (4,1);
\draw[ultra thick,red] (2,1) to (3,2);
\draw[ultra thick,red] (1,2) to (2,3);
\draw[ultra thick,red] (0,3) to (1,4);
\end{scope}
\end{tikzpicture}
\end{center}
\caption{$4$-diagonal steps in $\N^2$.}
\label{F:4steps}
\end{figure}

Next, we define the ``weighted Delannoy paths.''
\begin{Definition}
By a labelled step we mean a pair $(K,m)$,
where $K\in \{N,E,D\}$ and $m$ is a positive integer
such that $m=1$ if $K=N$ or $K=E$. 
A weighted $(p,q)$ Delannoy path is a word of the form 
$W:=K_1\dots K_r$, where $K_i$'s $(i=1,\dots, r)$ 
are labeled steps $K_i=(L_i,m_i)$ such that 
\begin{itemize}
\item $L_1\dots L_r$ is a Delannoy path from $\mc{D}(p,q)$;
\item if $L_i$ ($1\leq i \leq r$) is a $k$-th 
diagonal step, then $2\leq m_i \leq  2k-1$.
\end{itemize}
The set of all weighted $(p,q)$ Delannoy 
paths is denoted by $\mc{D}^w(p,q)$. 
\end{Definition}

\begin{Theorem}\label{T:anotherpaththeorem}
There is a bijection between the set of weighted $(p,q)$ Delannoy 
paths and the set of ssymmetric $(2p,2q)$ clans. 
In particular, we have 
$$
c_{p,q} = \sum_{W\in \mc{D}^w(p,q)} 1.
$$
\end{Theorem}

There is much more to be said about the lattice 
path interpretation of the number of Borel orbits 
in $\mathbf{Sp}_n/\mathbf{Sp}_p\times \mathbf{Sp}_q$
but we postpone them to a future paper. 
We finish our introduction by giving a brief outline of our paper.
We divided our paper into two parts. Before 
starting the first part, in Section~\ref{S:NP} 
we introduce the background material and notation that we use in the sequel.
In particular, we review a bijection between clans and involutions
and we introduce the symmetric $(2p,2q+1)$ clans as well as ssymmetric
$(2p,2q)$ clans. 
We start Part I by analyzing the numbers $\beta_{k,p,q}$.
In Section~\ref{SS:kpairs} we prove our Theorem~\ref{T:first}
and in the following Section~\ref{SS:Recsforbetas} we derive
aforementioned recurrences for $\beta_{k,p,q}$'s.
We devoted Section~\ref{SS:proof1} to the proof of
Theorem~\ref{T:second}. 
The Part II of our paper starts with an analysis of 
the numbers $\gamma_{k,p,q}$. 
In Section~\ref{SS:Symplectickpairs}, we prove 
the formula of $\gamma_{k,p,q}$'s as 
given in (\ref{A:gamma kqp1}). 
In the following Section~\ref{SS:Recsforgammas},
by developing the recurrences for these numbers
we prove Proposition~\ref{P:First recurrence}.
Section~\ref{SS:proofofmain2} is devoted to the proof
of Theorem~\ref{T:main2}. In particular, we point out in 
Remark~\ref{R:expensive} that it is possible to find a formula 
for the generating function of $c_{p,q}$'s 
at the expense of a very complicated expression.
In the remaining of Part II, we investigate the combinatorial 
interpretations of $c_{p,q}$'s. In Section~\ref{S:Interpretation},
we prove Proposition~\ref{P:anotherpaththeorem}
and Theorem~\ref{T:anotherpaththeorem}. Finally, 
in the appendix, which is Section~\ref{S:Appendix}, 
we present our methods for solving ODE (\ref{A:mainODE}).

\section{Notation and preliminaries}\label{S:NP}

The notation $\N$ stands for the
set of natural numbers, which includes 0.
Let us treat $+$ and $-$ as 
two symbols rather than viewing 
them as arithmetic operations.
Throughout our paper the notation 
$\PP$ stands for the set $\{ +,-\}\cup \N$. 
The elements of $\PP$ are called symbols. 
When we want to make a distinction
between the symbols $\pm$ and the elements 
of $\N$, we call the latter by numbers, 
following their usual trait.

Let $n$ be a positive integer. 
The symmetric group of permutations on 
$[n]:=\{1,\dots, n\}$ is denoted by $S_n$. 
If $\pi \in S_n$, then its {\em one-line} notation 
is the string $\pi_1\pi_2\dots \pi_n$,
where $\pi_i = \pi(i)$ ($i=1,\dots, n$). 
We trust that our reader is familiar with
the most basic terminology about 
permutations such as their cycle 
decomposition, cycle type, etc.. 
However, just in case, let us mention that 
the cycle decomposition $C_1\cdots C_r$
of a permutation 
is called standard if the entries of $C_i$'s
are arranged in such a way that 
$c_1< c_2 < \cdots < c_r$, where $c_i$ 
is the smallest number that appears in $C_i$
($i=1,\dots, r$). Here, we followed the common 
assumption that each cycle has at least two 
entries. Since we need the data of 
fixed points of a permutation, we will append to the cycle decomposition
$C_1\cdots C_r$ the one-cycles in 
an increasing order without using 
parentheses as indicated in the following example.
\begin{Example}
$(2,6,8)(4,5,7,9)13$ is the standard cycle 
decomposition of the permutation $\pi$ from $S_9$
whose one-line notation is given by $\pi=163578924$.
\end{Example}

Clearly, a permutation $\pi$ is an involution,
that is to say $\pi^2=id$, if and only if every
cycle of $\pi$ is of length at most 2.

\begin{Definition}
Let $p$ and $q$ be two positive integers and set 
$$n:=p+q.$$
Suppose that $p>q$. A $(p,q)$ preclan,
denoted by $\gamma$, is a string  
of symbols such that 
\begin{enumerate}
\item there are $p-q$ more $+$'s than $-$'s;
\item if a number appears in $\gamma$, 
then it appears exactly twice. 
\end{enumerate} 
In this case, we call $n$ the order of $\gamma$. 
For example, $1221$ is a $(2,2)$ preclan of order 4
and $+1+++-1$ is a $(5,2)$ preclan of order 7. 
We call two $(p,q)$ preclans $\gamma$ and $\gamma'$ 
equivalent if the positions of the matching numbers are the 
same in both of them. 
For example, $\gamma:=1221$ and $\gamma':=2112$ are 
in the same equivalence class of $(2,2)$ preclans 
since both of $\gamma$ and $\gamma'$ have matching 
numbers in the positions $(1,4)$ and $(2,3)$. 
Finally, we call an equivalence class of $(p,q)$ preclans
a $(p,q)$ clan. 
\end{Definition}

In most places in our paper, we will 
abuse the notation and represent 
the equivalence class of a preclan 
$\gamma$ by $\gamma$ also, however,  
it is sometimes useful not to do that. 
When we need to distinguish between 
the preclan and its equivalence
class we will use $\gamma$
and $[\gamma]$, respectively,
for the preclan and 
its equivalence class.
The order of a clan is
defined in the obvious way
as the order of any preclan
that it contains.

\begin{Lemma}\label{L:clanstoinvolutions}
There exists a surjective map from
the set of clans of order $n$ to the 
set of involutions in $S_n$, the symmetric 
group of permutations on $\{1,\dots, n\}$.
\end{Lemma}

\begin{proof}
Let $\gamma = c_1\cdots c_n$ be a 
(pre)clan of order $n$. 
For each
pair of identical numbers $(c_i,c_j)$ with $i<j$ 
we have a transposition in $S_n$ which is defined 
by the indices, that is $(i,j)\in S_n$. Clearly, if $(c_i,c_j)$ and 
$(c_{i'},c_{j'})$ are two pairs of identical numbers
from $\gamma$, then $\{i,j\} \neq \{i',j'\}$. 
Now we define the involution $\pi=\pi(\gamma)$ corresponding to 
$\gamma$ as the product of all transpositions
that come from $\gamma$. 
Accordingly, the $\pm$'s in $\gamma$ correspond to 
the fixed points of the involution $\pi$.

Conversely, if $\pi$ is an involution 
from $S_n$, then we have a $(p,q)$ preclan 
$\gamma=\gamma(\pi)$ that is defined
as follows. We start with an empty string 
$\gamma = c_1\dots c_n$ 
of length $n$.
If $ \pi_1\dots \pi_n$ 
is the one-line notation for $\pi$, then
for each pair of numbers $(i,j)$ such that $1\leq i < j \leq n$
and $\pi_i=j$, $\pi_j=i$, we put $c_i = c_j = i$. 
Also, if $ i_1,\dots, i_m$ is the increasing list of 
indices such that $\pi_{i_j} = i_j$ ($j=1,\dots, m$),
then starting from $i_1$ place a $+$ until 
the difference between the number of 
$c_{i_j}$'s with a + and the number of empty places 
is $p-q$. At this point place a $-$ in each of 
the empty places. It is easy to check 
that $\gamma$ is a $(p,q)$ preclan of order $n$,
hence the proof follows. 
\end{proof}

\begin{Definition}
Let $p$ and $q$ be two positive integers with $p>q$
and let $n:=p+q$.
A signed $(p,q)$ involution is an involution 
$\pi$ from $S_n$ whose fixed points are
labeled either by $+$ or by $-$ in such a 
way that the number of $+$'s is $p-q$ more 
than the number of $-$'s. 
\end{Definition}

\begin{Lemma}\label{L:clanstoinvolutions2}
There is a bijection between 
the set of all $(p,q)$ clans and 
the set of all signed $(p,q)$ involutions.
\end{Lemma}
\begin{proof}
Let $\varphi$ denote the surjection that 
is constructed in the proof of Lemma~\ref{L:clanstoinvolutions}.
We modify $\varphi$ as follows. Let $\gamma=c_1\dots c_n$
be a $(p,q)$ clan and let $\pi = \varphi(\gamma)$ 
denote involution that is obtained from $\gamma$ via $\varphi$. 
If an entry $c_i$ of $\gamma$ is a $\pm$, then we know that $i$ is a fixed 
point of $\pi$. We label $i$ with $\pm$. Repeating this 
procedure for each $\pm$ that appear in $\gamma$
we obtain a signed $(p,q)$ involution $\widetilde{\pi}$. 
Clearly $\widetilde{\pi}$ is uniquely determined by 
$\gamma$. Therefore, 
the map defined by $\widetilde{\varphi} (\gamma) = \widetilde{\pi}$
is a bijection. 
\end{proof}

Let $\gamma$ be a 
preclan of the form $\gamma=c_1\cdots c_n$.
The {\em reverse} of $\gamma$,
denoted by $rev(\gamma)$,
is the preclan 
$$
rev(\gamma) = c_nc_{n-1}\cdots c_1.
$$
Now we are ready to define 
the notion of a symmetric clan. 

\begin{Definition}
A $(p,q)$ clan $\gamma$ is called symmetric if 
$[\gamma ] =  [ rev(\gamma)]$.
\end{Definition}

\begin{Example}
The $(4,3)$ clan $\gamma = (1 2 + - + 1 2)$ is symmetric 
since the clan $(2 1 + - + 2 1)$ which is obtained from $\gamma$ 
by reversing its symbols is equal $\gamma$ as a clan. 
More explicitly, they are the same since both of them
have the same matching numbers in the positions $1,6$ and $2,7$.
\end{Example}

In our next example, we list all symmetric $(4,3)$ clans. 
\begin{Example}
\begin{align*}
\{{}&  123,+321,12+-+21, 123+312, 312+123, 12+-+12, 1+2-2+1, 132+132, 311+223,\\
& 1+2-1+2, +12-21+, 1+-+-+1, 1,3,1+,2,3,2, 1-+++-1, 1+1-2+2,\\ 
&+12-12+, +1-+-1+,113+322, -1+++1-, +11-22+, 11+-+22,\\
& +-1+1-+, -+1+1+-, -++-++-,  +-+-+-+, ++---++ \}.
\end{align*}
\end{Example}

\begin{Definition}
An ssymmetric $(2p,2q)$ clan 
$\gamma = c_1 \dots c_{2n}$ is a symmetric 
clan such that $c_i \neq c_{2n+1-i}$ 
whenever $c_i$ is a number (that is $c_i$ is not a sign).
The cardinality of the set of all ssymmetric $(2p,2q)$ 
clans is denoted by $c_{p,q}$.
\end{Definition}
  
\begin{Example}
The set of all ssymmetric $(4,2)$-clans is given by 
\begin{align*}
\{{}&  12++12, 1+21+2, 1+12+2,+1212+,  +1122+, 11++22,\\
& -++++-,  +-++-+, ++--++ \}.
\end{align*}
\end{Example}

We finish our preliminaries section
by a remark/definition.
\begin{Remark}\label{R:matchingvspairs}
Let $\pi$ denote the signed $(p,q)$ involution
corresponding to a $(p,q)$ clan $\gamma$ under the 
bijection $\widetilde{\varphi}$ that is defined 
in the proof of Lemma~\ref{L:clanstoinvolutions2}. 
A matching pair of numbers in any preclan that 
represents $\gamma$ corresponds to a 2-cycle 
of $\pi$. We will call $\gamma$ a $(p,q)$ clan
with $k$ pairs if $\pi$ has exactly $k$ 2-cycles. 
\end{Remark}

\section{Part I: Counting symmetric clans}\label{S:PartI}

\subsection{Symmetric $(2p, 2q+1)$-clans with $k$ pairs.}\label{SS:kpairs}

Let $p$ and $q$ be two positive integers 
such that $1\leq q < p$. Although we will be dealing with $(2p,2q+1)$ clans, we 
denote $p+q$ by $n$. Accordingly the 
number of symmetric $(2p,2q+1)$ clans is denoted by $b_{p,q}$. 
\vspace{.5cm}

By the proof of Lemma~\ref{L:clanstoinvolutions2} we know 
that there is a bijection, denoted by $\widetilde{\varphi}$, between 
the set of all $(2p,2q+1)$ clans of order $2n+1$ and the set of all signed
$(2p,2q+1)$ involutions in $S_{2n+1}$. We have 
a number of simple observations regarding this bijection. 

First of all, we observe that if $\pi$ is a signed $(2p,2q+1)$ involution
such that $\widetilde{\varphi}(\gamma)=\pi$, where 
$\gamma$ is a symmetric $(2p,2q+1)$ clan, then the following 
holds true:
\begin{itemize}
\item if $(i,j)$ with $1\leq i < j \leq 2n+1$ is 
a 2-cycle of $\pi$, then $n+1\notin \{i,j\}$ and 
$(2n+2-j, 2n+2-i)$ is a 2-cycle of $\pi$ also.
\end{itemize}
Secondly, we see from its construction
that $\widetilde{\varphi}$ maps a 
$(2p,2q+1)$ clan with $k$ pairs to 
a signed $(2p,2q+1)$ involution with 
$k$ 2-cycles. 
Let us denote the set of all such involutions by
$I^{\text{ort}}_{k,p,q}$ and we define $\beta_{k,p,q}$ as the cardinality 
$$
\beta_{k,p,q}:= | I^{\text{ort}}_{k,p,q} |.
$$ 
\begin{Remark}\label{R:observe1}
If $\pi \in I^{\text{ort}}_{k,p,q}$, then in the corresponding clan 
there are $2p-2q-1$ more $+$'s than $-$'s. 
Notice that the inequality $2p-2q-1 \leq 2p+2q+1-2k$ implies that $0 \leq k \leq 2q+1$.
\end{Remark}
It follows from the note in Remark~\ref{R:observe1} and
the fact that $\widetilde{\varphi}$ is a bijection, 
the number of symmetric $(2p,2q+1)$ clans is given by 
$$
b_{p,q} = \sum_{l=0}^{q} (\beta_{2l,p,q} + \beta_{2l+1,p,q}).
$$
Our goal in this section is to record a formula 
for $b_{p,q}$ that depends only on $p$ and $q$. 
To this end, first we determine the number 
of $\pm$'s in a symmetric $(2p,2q+1)$ clan.

\begin{Lemma}\label{L:L1}
If $\gamma = c_1\dots c_{2n+1}$ is a symmetric $(2p,2q+1)$ clan,
then either $c_{n+1}=+$ or $c_{n+1}=-$. 
\end{Lemma}

\begin{proof}
First, assume that $\gamma$ has even number of pairs. 
Let $k$ denote this number, $k=2l$. 
Let $\alpha, \beta$, respectively, denote the number of $+$'s and $-$'s 
in $\gamma$. 
Then we have 
$$
\alpha + \beta = 2p+2q+1-4l \qquad \text{and} \qquad \alpha - \beta = 2p - 2q -1 .
$$
It follows that 
$$
\alpha  = 2p-2l \qquad \text{and} \qquad \beta = 2q-2l+1,
$$
so, in $\gamma$ there are odd number of $-$'s and there are even number of $+$'s.
As a consequence we see that $c_{n+1}$ is a $-$. 

Next, assume that $\gamma$ has an odd number of pairs, that is $k=2l+1$. 
Arguing as in the previous case we see that 
there is an odd number of $+$'s, hence $c_{n+1}$ is a +.
This finishes the proof.
 \end{proof}

We learn from the proof of Lemma~\ref{L:L1} that 
it is important to analyze the parity of pairs, so we 
record the following corollary of the proof for a 
future reference.

\begin{Corollary}\label{C:L1}
Let $k$ denote the number of pairs in a symmetric 
$(2p,2q+1)$ clan $\gamma$. 
If $k=2l$ ($0 \leq l \leq q$),
then the number of $+$'s in $\gamma$ is $2(p-l)$.
If $k=2l+1$ ($0\leq l \leq q$), then 
the number of $-$'s in $\gamma$ is $2(q-l)+1$.
\end{Corollary}

Our next task is determining the number of 
possible ways of placing $k$ pairs to build from 
scratch a symmetric $(2p,2q+1)$ clan
$$
\gamma=c_1\cdots c_n c_{n+1} c_{n+2}\cdots c_{2n+1} \qquad \text{(with $c_{n+1}=\pm$)}.
$$
To this end we start with defining some interrelated 
sets.
\begin{align*}
I_{1,1} :=& \{((i,j),(2n+2-j, 2n+2-i))\;|\; 1 \leq i < j \leq n \}, \\
I_{1,2} :=& \{((i,j),(2n+2-j, 2n+2-i))\;|\; 1 \leq i < n+1 < j \leq 2n+1 \}, \\
I_1 :=& I_{1,1} \cup I_{1,2},\\
I_2 :=& \{(i,j) \;|\; 1 \leq i < n+1 < j \leq 2n+1, \; i+j = 2n+2 \}.
\end{align*}
We view $I_1$ as the set of placeholders for two distinct pairs 
that determine each other in $\gamma$. 
The set $I_2$ corresponds to the list of 
stand alone pairs in $\gamma$. In other words, 
if $(i,j)\in I_2$, then $c_i = c_j$ and $j=2n+1-i+1$.

\begin{Example}
Let us show what $I_1$ and $I_2$ correspond to with a concrete 
example. If $\gamma$ is the symmetric $(4,3)$ clan 
$$
\gamma = (7\; 2 +\, 0\; 8+\, 9 -\; 8 +\, 9\; 0+\,7\; 2),
$$ 
then $I_{1,1} = \{ ((1,14),(2,15))\}$, $I_{1,2}=\{((5,9),(7,11)) \}$,
$I_2= \{(4,10)\}$.
\end{Example}

If $(c_i,c_j)$ is a pair in the symmetric clan $\gamma$
and if $(i,j)$ is an element of $I_2$, then we call $(c_i,c_j)$
a pair of type $I_2$. 
If $x$ is a pair of pairs of the form $((c_i,c_j),(c_{2n+2-j},c_{2n+2-i}))$
in a symmetric clan $\gamma$ and if $((i,j),(2n+2-j,2n+2-i))\in I_{1,s}$ ($s\in \{1,2\}$),
then we call $x$ a pair of pairs of type $I_{1,s}$.
If there is no need for precision, then we will call $x$
a pair of pairs of type $I_1$.

Clearly, if $|I_1| = b$ and $|I_2| = a$, then $2b + a= k$ is 
the total number of pairs in our symmetric clan $\gamma$. 
To see in how many different ways these pairs of indices can 
be situated in $\gamma$, we start with choosing $k$ spots
from the first $n$ positions in $\gamma = c_1\cdots c_{2n+1}$. 
Obviously this can be done in ${n\choose k}$ many different ways.
Next, we count different ways of choosing $b$ pairs 
within the $k$ spots to place the $b$ pairs of pairs of 
type $I_1$. This number of possibilities
for this count is ${k\choose {2b}}$. 
Observe that choosing a pair from $I_1$ is equivalent to choosing 
$(i,j)$ for the pairs of pairs in $I_{1,1}$ and choosing $(i, 2n+2-j)$ for the 
pairs of pairs in $I_{1,2}$.
More explicitly, we first choose $b$ pairs among the $2b$ elements 
and then place them on $b$ spots; this can be done in 
${{2b} \choose b} b!$ different ways.
Once this is done, 
finally, the remaining spots will be filled by the $a$ pairs 
of type $I_2$. This can be done in only one way. 
Therefore, in summary, the number of different ways 
of placing $k$ pairs to build a symmetric 
$(2p,2q+1)$ clan $\gamma$ is given by 
\begin{align*}
{n \choose k} \sum_{b=0}^{  \floor*{\frac{k}{2}}} {k \choose 2b}{2b \choose b} b!, 
\hspace{.5cm} \text{ or equivalently, }  \hspace{.5cm}
{n \choose k} \sum_{b=0}^{  \floor*{\frac{k}{2}}} {k \choose 2b} \frac{(2b)!}{b!}.
\end{align*}

In conclusion, we have the following preparatory result. 

\begin{Theorem}[Theorem~\ref{T:first}]
The number symmetric $(2p,2q+1)$ clans with $k$ pairs is given by 
\begin{align}\label{OS:first formula}
\beta_{k,p,q} =
\begin{cases}
 {{n-2l} \choose {p-l} }{n \choose 2l} a_{2l} & \text{ if } \; k=2l; \\
  {{n-(2l+1)} \choose {p-(l+1)}} {n \choose 2l+1} a_{2l+1}& \text{ if } \; k=2l+1,
\end{cases}
\end{align}
where 
\begin{align}\label{A:formulafora}
a_{2l} := \sum_{b=0}^{l} {2l \choose 2b} \frac{(2b)!}{b!}
\hspace{.75cm} \text{  and }  \hspace{.75cm}
a_{2l+1} := \sum_{b=0}^{l} {{2l+1} \choose 2b} \frac{(2b)!}{b!}.
\end{align}
Consequently, the total number of symmetric 
$(2p,2q+1)$ clans is given by
$$
b_{p,q} = 
\sum_{l=0}^q  \left\lbrack {{n-2l} \choose {p-l} }{n \choose 2l} a_{2l}
+  {{n-(2l+1)} \choose {p-(l+1)}} {n \choose 2l+1}a_{2l+1}\right\rbrack.
$$
\end{Theorem}

\begin{proof}
As clear from the statement of our theorem, we 
will consider the two cases where $k$ is even and where $k$ is odd separately. 
We already computed the numbers of possibilities for placing 
$k$ pairs, which are given by $a_{2l}$ and $a_{2l+1}$,
 but we did not finish counting the number of possibilities
for placing the signs. 

\begin{enumerate}
\item $k = 2l$ for $0 \leq l \leq q$. In this case, by Lemma \ref{L:L1},
we see that the number of $+$ signs is $\alpha:= 2p -2l=2(p-l)$.
Notice that because of symmetry condition it is enough to focus 
on the first $n$ spots to place $\pm$ signs. 
Thus, there are ${n - 2l  \choose p-l }$ possibilities to place $\pm$ signs.

\item $k=2l+1$ for $0 \leq l \leq q$. In this case, it follows from 
Lemma \ref{L:L1} that the entry in the $(n+1)$-th place is $+$. 
By using an argument as before, we see that 
there are ${ n - (2l +1)  \choose p-(l+1)}$ possibilities to place $\pm$ signs. 
\end{enumerate}

This finishes the proof. 
\end{proof}

The formula for $b_{p,q}$ that is derived in Theorem~\ref{T:first}
is not optimal in the sense that it is hard to write down a closed 
form of its generating function this way. Of course, the complication
is due to the form of  $\beta_{k,p,q}$, where $k$ is even or odd. 
Both of the cruces are resolved by considering the recurrences;
we will present our results in the next subsection.

\subsection{Recurrences for $\beta_{k,p,q}$'s.}\label{SS:Recsforbetas}

We start with some easy recurrences.
\begin{Lemma}\label{L:4identity}
Let $p$ and $q$ be two positive integers, 
and $l$ be a nonnegative integer. 
In this case, whenever both sides of the following equations
are defined, they hold true:
\begin{align}\label{A:simplerecs}
\beta_{2l,p-1,q} &= \frac{p-l}{p+q} \beta_{2l,p,q},	\\ 
\beta_{2l,p,q-1} &= \frac{q-l}{p+q} \beta_{2l,p,q},	\\
\beta_{2l+1,p-1,q} &= \frac{p-l-1}{p+q} \beta_{2l+1,p,q},	\\
\beta_{2l+1,p,q-1} &= \frac{q-l}{p+q}\beta_{2l+1,p,q}.
\end{align}
\end{Lemma}
The proofs of the identities in Lemma~\ref{L:4identity} 
follow from obvious binomial identities and 
our formulas in Theorem~\ref{T:first}. 
But note that $l$ does not change in them. 
In the sequel, we will find other recurrences that 
run over $l$'s. Towards this end, the following lemma,
whose proof is simple, will be useful.

\begin{Lemma}
Let $a_k$ denote the numbers as in (\ref{A:formulafora}).
If $k\geq 2$, then we have 
\begin{align}\label{A:eorec}
a_k = a_{k-1} + 2 (k -1) a_{k-2}.
\end{align}
\end{Lemma}

By using (\ref{A:eorec}) we find relations between $\beta_{k,p,q}$'s. 
Let $k$ be an even number of the form $k = 2l$. Then we find that 
\begin{align}
\beta_{2l,p,q} =& {{n-2l} \choose{p-l}} {n \choose 2l} a_{2l}  \notag \\
=& {{n-2l} \choose{p-l}} {n \choose 2l} (a_{2l-1} + 2(2l-1) a_{2l-2})  \notag\\
=& {{n-2l} \choose{p-l}} {n \choose 2l} a_{2l-1} + 2(2l-1){{n-2l} \choose{p-l}} {n \choose 2l} a_{2l-2} \notag\\
=& \frac{n-2l+1 -p+l}{n-2l+1}{{n-(2l-1)} \choose{p-l}} \frac{n+1 -2l}{2l}{n \choose{2l-1}} a_{2l-1} 
\notag\\
+& 2(2l-1) \frac{(p-l+1)(n-l+1-p)}{(n-2l+2)(n-2l+1)}{{n-(2l-2)} \choose{p-l}} \frac{(n-2l+2)(n+1-2l)}{(2l)(2l-1)}{n \choose{2l-2}} a_{2l-2} \notag\\
=& \frac{n-l-p+1}{2l} \beta_{2l-1,p,q} + 2 \frac{(p-l+1)(n-l+1-p)}{2l} \beta_{2l-2,p,q} \notag\\
=& \frac{q-l+1}{2l} \beta_{2l-1,p,q} + 2 \frac{(p-l+1)(q-l+1)}{2l} \beta_{2l-2,p,q}. \label{A:kevenfr}
\end{align}
In a similar manner, for an odd number of the form $k=2l+1$, 
we find that 
\begin{align}
\beta_{2l+1,p,q} =& {{n-2l-1} \choose{p-l-1}} {n \choose {2l+1}} a_{2l+1} \notag\\
=& {{n-2l-1} \choose{p-l-1}} {n \choose {2l+1}} (a_{2l} + 2(2l) a_{2l-1}) \notag\\
=& {{n-2l-1} \choose{p-l-1}} {n \choose {2l+1}} a_{2l} + 2(2l){{n-2l-1} \choose{p-l-1}} {n \choose {2l+1}} a_{2l-1} \notag\\	
=& \frac{p-l}{n-2l}{{n-2l} \choose{p-l}} \frac{n+1 -(2l+1)}{2l+1}{n \choose{2l}} a_{2l} \notag\\
+& 2(2l) \frac{(p-l)(p-l+1)}{(n-2l)(n-2l+1)}{{n-2l+1} \choose{p-l+1}} \frac{(n-2l)(n+1-2l)}{(2l)(2l+1)}{n \choose{2l-1}} a_{2l-1} \notag\\
=& \frac{p-l}{2l+1} \beta_{2l,p,q} + 2 \frac{(p-l)(q-l+1)}{2l+1} \beta_{2l-1,p,q}.\label{A:koddfr}
\end{align}

Now we two recurrences (\ref{A:kevenfr}) and (\ref{A:koddfr})
mixing the terms $\beta_{k,p,q}$ for even and odd $k$. 
To separate the parity, we rework on our initial recurrence (\ref{A:eorec}).


\begin{Lemma}\label{L:tworecsfora}
For all $1 \leq l \leq q-1$, the following recurrences:
\begin{align}
a_{2l+2} =& (8l+3)a_{2l} + 4 (2l)(2l -1) a_{2l-2} \label{A:erec} \\
a_{2l+3} =& (8l+7)  a_{2l+1} + 4 (2l+1)(2l) a_{2l-1} \label{A:orec} 
\end{align}
with $a_0=1, a_1 =1$ are satisfied.
\end{Lemma}
\begin{proof}
We will give a proof for the former equation here. 
The latter can be proved in a similar way.

We start with splitting (\ref{A:eorec}) 
into two recurrences:
\begin{align}
a_{2l+1} =& a_{2l} + 2 (2l) a_{2l-1} \label{A:O} \\
a_{2l} =& a_{2l-1} + 2 (2l-1) a_{2l-2}. \label{A:E}
\end{align}
On one hand it follows from eqn. \ref{A:E} that we have
\begin{align*}
a_{2l-1} =& a_{2l} - 2 (2l-1) a_{2l-2}.
\end{align*}
Plugging this into eqn. \ref{A:O} yields 
\begin{align*}
a_{2l+1} =& a_{2l} + 2 (2l)( a_{2l} - 2(2l-1)a_{2l-2}) \ \text{or} \\
a_{2l+1} =& (1 + 2(2l))a_{2l} - 4(2l)(2l-1)a_{2l-2}).
\end{align*}
On the other hand, we know that 
\begin{align*}
a_{2l+2}= a_{2l+1} + 2(2l+1) a_{2l}.
\end{align*}
If we plug this into the previous equation, then we obtain
\begin{align*}
a_{2l+2} =& (1+2(2l))a_{2l} - 4 (2l)(2l-1)a_{2l-2} + 2(2l+1)a_{2l} \\
=& (8l+3)a_{2l} - 4(2l)(2l-1)a_{2l-2} \qquad (1 \leq l \leq q-1),
\end{align*}
which finishes the proof of our claim.
\end{proof}


Next, by the help of Lemma~\ref{L:tworecsfora}, we 
obtain a recurrence relation for $\beta_{k,p,q}$'s 
where all of $k$'s are even numbers. 
\begin{align}
\beta_{2l+2,p,q} =& {{n-2l-2} \choose{p-l-1}} {n \choose {2l+2}} a_{2l+2} \notag\\
=& {{n-2l-2} \choose{p-l-1}} {n \choose {2l+2}} ((8l+3)a_{2l} - 4(2l)(2l-1)a_{2l-2})  \notag\\
=& (8l+3){{n-2l-2} \choose{p-l-1}} {n \choose 2l+2} a_{2l} + 4 (2l)(2l-1){{n-2l-2} \choose{p-l-1}} {n \choose 2l+2} a_{2l-2} \notag\\
=& (8l+3)\frac{(p-l)(q-l)}{(n-2l)(n-2l-1)}{{n-2l} \choose{p-l}} \frac{(n-2l)(n-2l-1)}{(2l+2)(2l+1)}{n \choose{2l}} a_{2l} \notag\\
-& 4(2l)(2l-1) \frac{(p-l)(p-l+1)(q-l)(q-l+1)}{(n-2l+2)(n-2l+1)(n-2l)(n-2l-1)}{{n-2l} \choose{p-l+1}}\notag\\
 {}&\frac{(n-2l+2)(n+1-2l)(n-2l)(n-2l-1)}{(2l+2)(2l+1)(2l)(2l-1)}{n \choose{2l-2}} a_{2l-2} \notag\\
=&(8l+3) \frac{(p-l)(q-l)}{(2l+2)(2l+1)} \beta_{2l,p,q} - 4 \frac{(p-l)(p-l+1)(q-l)(q-l+1)}{(2l+2)(2l+1)} \beta_{2l-2,p,q}. \label{A:kevensr}
\end{align}
The proof of the following recurrence follows from similar arguments.
\begin{align}\label{A:koddsr}
\beta_{2l+3,p,q} = (8l+7) \frac{(q-l)(p-l-1)}{(2l+3)(2l+2)} \beta_{2l+1,p,q} - 4 \frac{(p-l)(p-l-1)(q-l)(q-l+1)}{(2l+3)(2l+2)} \beta_{2l-1,p,q}.
\end{align}

\subsection{The proof of Theorem~\ref{T:second}.}\label{SS:proof1}

As we mentioned in the introduction, we are looking for the closed form of 
the generating function 
$$
B(y,z) = \sum_{p\geq 0} B_p(1,y)z^p,
$$
where 
\begin{align*}
b_{p,q}(x) = \sum_{l=0}^{q} (\beta_{2l,p,q} x^{q-l} + \beta_{2l+1,p,q}x^{q-l} ) \hspace{.3cm}
\text{ and } \hspace{.3cm} B_p(x,y) = \sum_q b_{p,q}(x) y^q.
\end{align*}
In particular, we are looking for an expression of $b_{p,q}(1)$ which is simpler than 
the one that is given in Theorem~\ref{T:first}.

Obviously, 
$$
b_{p,q-1}(x) = \sum_{l=0}^{q-1} (\beta_{2l,p,q-1} x^{q-l-1} + \beta_{2l+1,p,q-1} x^{q-l-1}).
$$
It follows from Lemma~\ref{L:4identity}
that 
\begin{align*}
b_{p,q}(x) =& (\beta_{2q,p,q} + \beta_{2q+1,p,q})x^0 + \sum_{l=0}^{q-1} (\beta_{2l,p,q} x^{q-l} + \beta_{2l+1,p,q} x^{q-l})\\
=&(\beta_{2q,p,q+1} + \beta_{2q+1,p,q+1}) + \sum_{l=0}^{q-1} (p+q) \biggr(\frac{\beta_{2l,p,q-1}}{q-l} x^{q-l} + \frac{\beta_{2l+1,p,q-1}}{q-l}x^{q-l} \biggr).
\end{align*}
Taking the derivative of both sides of the above equation gives us that 
\begin{align*}
b_{p,q}^{'} (x) =& \sum_{l=0}^{q-1} (p+q) \biggr( \beta_{2l,p,q-1} x^{q-l-1} + \beta_{2l+1,p,q-1}x^{q-l-1} \biggr),
\end{align*}
or, equivalently, gives that 
\begin{align}\label{A:derr}
b_{p,q}^{'} (x) = (p+q) b_{p,q-1} (x).
\end{align}
The differential equation (\ref{A:derr}) leads to  
a PDE for our initial generating function $B_p(x,y)$:
\begin{align*}
\frac{\partial}{\partial x} (B_p(x,y)) 
=& \frac{\partial}{\partial x} \biggr[\sum_{q \geq 0} b_{p,q} (x)  y^q \biggr]
=b_{p,0}^{'}y^0 + \sum_{q \geq 1} b_{p,q}^{'} (x)  y^q \ \qquad (b_{p,0}^{'} =0 ) \\
=&  \sum_{q \geq 1}(p+q) b_{p,q-1}(x) y^q
= py \sum_{q \geq 1}  b_{p,q-1}(x)y^{q-1}+ y\sum_{q \geq 1}  q b_{p,q-1}(x) y^{q-1} \\
=& pyB_p(x,y)+ y \biggr(\frac{\partial}{\partial y} (y \cdot B_p(x,y)) \biggr) \\
=& y^2\frac{\partial}{\partial y} B_p(x,y) + y B_p(x,y) + py B_p(x,y).
\end{align*}
By the last equation we obtain the PDE that we mentioned in the introduction: 
\begin{align}\label{A:fundamentalPDE1}
\frac{\partial}{\partial x} B_p(x,y) - y^2\frac{\partial}{\partial y} B_p(x,y)= y(1+p) B_p(x,y).
\end{align}

The general solution $S(x,y)$ 
of (\ref{A:fundamentalPDE1}) is given by 
\begin{align}
S(x,y)=\frac{1}{y^{p+1}}G\left(\frac{1-xy}{y}\right),
\end{align}
where $G(z)$ is some function in one-variable. 
We want to choose $G(z)$ in such a way that $S(x,y)=B_p(x,y)$
holds true. To do so, first, we look at some special values of $B_p(x,y)$.

If let $x=0$, then $B_p(0,y) = \sum_{q\geq 0} b_{p,q}(0) y^p$ and 
$b_{p,q}(0) = 2(\beta_{2q,q,p}+\beta_{2q+1,q,p})$ for all $q> 0$.
Also, recall from the introduction that if $q=0$, then $b_{p,q}=p+1$. 
Thus, we ask from $G(z)$ that it satisfies the following equation  
\begin{align*}
\frac{1}{y^{p+1}}G\left(\frac{1}{y}\right) = (p+1) + 2 \sum_{q\geq 1} 
(\beta_{2q,q,p}+\beta_{2q+1,q,p}) y^q,
\end{align*}
or that 
\begin{align}\label{A:ask}
G\left(\frac{1}{y}\right) = y^{p+1} \left( (p+1) + 2 \sum_{q\geq 1} 
(\beta_{2q,q,p}+\beta_{2q+1,q,p}) y^q\right).
\end{align}

Therefore, we see that our generating function is given by 
\begin{align}
B_p(x,y) &= \frac{1}{y^{p+1}}G\left(\frac{1}{y/(1-xy)}\right) \notag \\
&= 
\frac{1}{y^{p+1}}
\left( \frac{y}{1-xy}\right)^{p+1} \left( (p+1) + 2 \sum_{q\geq 1} 
(\beta_{2q,q,p}+\beta_{2q+1,q,p}) \left( \frac{y}{1-xy}\right)^q\right) \notag \\
&= 
\left( \frac{1}{1-xy}\right)^{p+1} \left( (p+1) + 2 \sum_{q\geq 1} 
(\beta_{2q,q,p}+\beta_{2q+1,q,p}) \left( \frac{y}{1-xy}\right)^q\right). \label{A:prep}
\end{align}
\vspace{.5cm}

To get a more precise information about $b_{p,q}$'s we substitute $x=1$ 
in (\ref{A:prep}):
\begin{align*}
B_p(1,y) 
&= \frac{1}{(1-y)^{p+1}} \left( (p+1) + 2 \sum_{q\geq 1} 
(\beta_{2q,q,p}+\beta_{2q+1,q,p}) \left( \frac{y}{1-y}\right)^q\right),
\end{align*}
or 
\begin{align}\label{A:prep1}
(1-y)^{p+1} B_p(1,y) 
&=(p+1) + 2 \sum_{q\geq 1} 
(\beta_{2q,q,p}+\beta_{2q+1,q,p}) \left( \frac{y}{1-y}\right)^q.
\end{align}

Now we apply the transformation $y\mapsto z= y/(1-y)$ in (\ref{A:prep1}):
\begin{align}
\frac{ 1}{(1+z)^{p+1}} B_p\left(1,\frac{z}{1+z}\right) &= 
 (p+1) + 2 \sum_{q\geq 1} 
(\beta_{2q,q,p}+\beta_{2q+1,q,p})z ^q.
\end{align}
This finishes the proof of Theorem~\ref{T:second}
since $B_p\left(1,\frac{z}{1+z}\right) = f_p(z)$.


\section{Part II: Counting ssymmetric clans}

\begin{Convention}
For this part of our paper, without loss of generality, we assume that $p$
and $q$ are nonnegative integers such that $p\geq q$.
\end{Convention}

\subsection{Ssymmetric clans with $k$-pairs.}\label{SS:Symplectickpairs}

Recall that a ssymmetric $(2p,2q)$ clan 
$\gamma = c_1 \dots c_{2n}$ is a symmetric 
clan such that $c_i \neq c_{2n+1-i}$ 
whenever $c_i$ is a number. 
In this second part of our paper, we are going to
find various generating functions and combinatorial
interpretations for the number $c_{p,q}$ of 
ssymmetric $(2p,2q)$ clans. 
We start by stating a simple lemma that 
tells about the involutions corresponding to 
ssymmetric clans.

\begin{Lemma}\label{L:evennumberof}
Let $\gamma = c_1 c_2 \dots c_{2n}$ be a ssymmetric 
$(2p,2q)$ clan. If $\pi \in S_{2n}$ is the associated involution
with $\gamma$, then there are even number of 2-cycles in $\pi$.
\end{Lemma}
\begin{proof}
First, notice that if for some $1 \leq i < j \leq 2n$ 
the numbers $c_i$ and $c_j$ form a pair, that is to say a 2-cycle in $\pi$, 
then by symmetry  $c_{2n+1-i}$ and $c_{2n+1-j}$ form a pair in $\pi$ as well. 
In addition, by the condition that is requiring for all natural $c_i$'s that 
$c_i \neq c_{2n+1-i}$, $c_i$ and $c_{2n+1-j}$ cannot form a pair in $\pi$. 
Therefore, if we have a pair $(c_i,c_j)$ in $\pi$, then we must also have another pair 
$(c_{2n+1-j}, c_{2n+1-i})$ which is different from $(c_i, c_j)$. 
Said differently, the number of 2-cycles in $\pi$ must be even. 
\end{proof}

In the light of Lemma~\ref{L:evennumberof},  
we will focus on the subset $I^{sp}_{k,p,q} \subset S_{2n}$ consisting 
of involutions $\pi$ whose standard cycle decomposition  
is of the form 
$$
\pi = (i_1 j_1) \dots (i_{2k} j_{2k}) d_1 \dots d_{2n-4k}.
$$
Furthermore, we assume the fixed points of $\pi$ are labeled 
by the elements of $\{ + , -\}$ in such a way that 
there are $2p-2q$ more $+$'s than $-$'s
and we want the following conditions be satisfied: 
\begin{enumerate}
\item $k\leq q$ (this is because there are $2p-2q$ more $+$'s than $-$'s, hence $2q+2p-4k \geq 2p-2q$);
\item if $(i,j)$ is a 2-cycle such that $1 \leq i < j \leq n$, then $(2n+1-j, 2n+1-i)$ is a 2-cycle also; 
\item if $(i,j)$ is a 2-cycle such that $1 \leq i < n+1 \leq j \leq 2n$, then $(2n+1-j, 2n+1-i)$ is a 2-cycle as well. 
\end{enumerate}

The (signed) involutions in $I^{sp}_{k,p,q}$ are 
precisely the involutions that correspond to the ssymmetric $(2p,2q)$ clans
under the bijection of Lemma~\ref{L:clanstoinvolutions2},
so, $\gamma_{k,p,q}$ stands for the cardinality of $I^{sp}_{k,p,q}$.
To find a formula for $\gamma_{k,p,q}$'s we argue 
similarly to the case of $\beta_{k,p,q}$, by counting 
the number of possible ways of placing pairs and by
counting the number of possible ways of placing $\pm$'s 
on the fixed points. 
Also, we make use of the bijection $\widetilde{\varphi}$
of Lemma~\ref{L:clanstoinvolutions2} to switch 
between the involution notation and the clan notation.

First of all, an involution $\pi$ from $I^{sp}_{k,p,q}$
has $2k$ 2-cycles and $2n-4k$ fixed points. 
The $2k$ 2-cycles, by using numbers from $\{1,\dots, 2n\}$
can be chosen in ${ n \choose 2k}$; the number of 
rearrangements of these $2k$ pairs and their entries, 
to obtain the standard form of an involution, requires $\frac{(2k)!}{k!}$ 
steps. In other words, the 2-cycles 
of $\pi$ are found and placed in the standard ordering
in ${ n \choose 2k } \frac{(2k)!}{k!}$ possible ways. 
Once we have the 2-cycles of the involution, 
we easily see that the numbers and their positions
in the corresponding ssymmetric clan are uniquely 
determined. 

Next, we determine the number of ways to place
$\pm$'s. This amounts 
to finding the number of ways of placing 
$2\alpha$ +'s and $2\beta$ -'s on the string $d_1 \dots d_{2n-4k}$ 
so that there are exactly $2p-2q=2\alpha-2\beta$ +'s more than -'s. 
By applying the inverse of the bijection $\widetilde{\varphi}$ 
of Lemma~\ref{L:clanstoinvolutions2}, we will use the symmetry 
condition on the corresponding clan. Thus, we observe that 
it is enough to focus on the first $n$ places of the clan only. 
Now, the number of $+$'s in the first $n$ places can 
be chosen in ${ n- 2k \choose \alpha }$ different ways. 
Once we place the $+$'s, the remaining entries will be filled
with $-$'s. Clearly there is now only one way of doing this 
since we placed the numbers and the $+$ signs already. 
Therefore, to finish our counting, we need to find what that $\alpha$ is.
Since $\alpha+\beta = n-2k= q+p-2k$ and since $\alpha-\beta = p-q$, 
we see that $\alpha = p-k$.

In summary, the number of possible ways of constructing 
a signed involution corresponding to a ssymmetric $(2p,2q)$ clan 
is given by 
\begin{align}\label{A:first formula}
\gamma_{k,p,q} = { q+p \choose 2k } \frac{ (2k)! }{ k! } { q+ p - 2k  \choose p-k }.
\end{align}
Note here that we are using $n=p+q$. 
The right-hand side of (\ref{A:first formula}) can be expressed 
more symmetrically as follows: 
\begin{align}\label{A:gamma kqp}
\gamma_{k,p,q} = \frac{(q+p)! }{ (q-k)! (p-k)! k!}.
\end{align}

\subsection{Recurrences for $\gamma_{k,p,q}$'s.}\label{SS:Recsforgammas}

Observe that the formula (\ref{A:gamma kqp}) is defined independently of the inequality 
$q\leq p$. 
From now on, for our combinatorial purposes, we skip mentioning this comparison between $p$ and $q$ and 
use the equality $\gamma_{k,p,q}=\gamma_{k,q,p}$ whenever it is needed. 
Also, we record the following obvious recurrences for future reference:
\begin{align}
\gamma_{k,p,q} &= \frac{(p-k+1)(q-k+1)}{k} \gamma_{k-1,p,q},\label{A:Future1} \\
\gamma_{k,p-1,q} &=  \frac{p-k}{p+q}  \gamma_{k,p,q}, \label{A:Future2}\\
\gamma_{k,p,q-1} &= \frac{q-k}{p+q} \gamma_{k,p,q}. \label{A:Future3}
\end{align}
These recurrences hold true whenever both sides of the equations 
are defined. Notice that in (\ref{A:Future1})--(\ref{A:Future3}) 
the parity, namely $k$ does not change. 
Next, we will show that $\gamma_{k,p,q}$'s obey a 3-term recurrence
once we allow change in all three numbers $p,q$, and $k$.

\begin{Lemma}\label{T:rec}
Let $p$ and $q$ be two positive integers. If $k\geq 1$, 
then we have 
\begin{equation}\label{recurrence}
\gamma_{k,p,q} =  \gamma_{k,p-1,q} +  
\gamma_{k,p,q-1} + 2(q+p-1) \gamma_{k-1,p-1,q-1}\; \; \text{and} \;\; 
\gamma_{0,p,q}= {p+q \choose p}.
\end{equation}
\end{Lemma}

\begin{proof}
Instead of proving our result directly, we will make use
of a similar result that we proved before.
Let $\widetilde{\gamma}_{k,p,q}$ denote 
the number 
\begin{align}\label{A:gamma kqp}
\widetilde{\gamma}_{k,p,q} = \frac{(q+p)! }{ 2^k (q-k)! (p-k)! k!}.
\end{align}
In \cite{CU}, it is proven that 
\begin{align}\label{A:previousgammarec}
\widetilde{\gamma}_{k,p,q} =  \widetilde{\gamma}_{k,p-1,q} +\widetilde{\gamma}_{k,p,q-1}
+(p+q-1) \widetilde{\gamma}_{k-1,p-1,q-1}
\end{align}
holds true for all $p,q,k\geq 1$. 
Note that $\widetilde{\gamma}_{0,p,q}= {p+q \choose p}$, which is our initial condition
for $\gamma_{k,p,q}$'s. Therefore, 
combining (\ref{A:previousgammarec}) with the fact that 
$\gamma_{k,p,q} = 2^k \widetilde{\gamma}_{k,p,q}$ finishes our proof. 
\end{proof}

\begin{Convention}\label{Con2}
From now on we will assume that $c_{p,q}=1$ whenever 
one or both of $p$ and $q$ are zero. 
\end{Convention}

\begin{Proposition}[Proposition~\ref{P:First recurrence}]
For all positive integers $p$ and $q$, the following recurrence
relation holds true: 
\begin{equation}\label{A:cpqrecurrence}
c_{p,q} =  c_{p-1,q} + c_{p,q-1} +  2(p+q-1) c_{p-1,q-1}.
\end{equation}
\end{Proposition}

\begin{proof}
Recall that $c_{p,q} = \sum_k \gamma_{k,p,q}$. 
Thus, summing both sides of eqn (\ref{recurrence}) over $k$ 
with $1 \leq k \leq p-1$ gives 
\begin{align*}
c_{p,q} -  c_{p-1,q} -  c_{p,q-1} - 2(p+q-1) c_{p-1,q-1} 
&= \gamma_{0,p,q} -  \gamma_{0,p-1,q} -  \gamma_{0,p,q-1} \\
&+ \gamma_{p,p,q} -  \gamma_{p,p,q-1} - 2(p+q-1) \gamma_{p-1,p-1,q-1} \\
&= 0.
\end{align*} 
\end{proof}

\subsection{Proof of Theorem~\ref{T:main2}}\label{SS:proofofmain2}

One of the many options for a bivariate 
generating function for $c_{p,q}$'s is the following 
\begin{align}\label{A:one of many}
v(x,y) := \sum_{p,q \geq 0 } c_{p,q} \frac{ (2x)^q y^p}{p!}.
\end{align}
Let us tabulate first few terms of $v(x,y)$:
\begin{align}
\sum_{p,q \geq 0} c_{p,q} \frac{ (2x)^q y^p}{p!}
&= c_{0,0} + c_{0,1} 2x + \dots + c_{0,q} (2x)^q +\dots \notag\\
&+ \frac{c_{1,0}}{1!} y + \dots + \frac{c_{p,0}}{p!} y^p + \dots \notag \\
&+ \frac{c_{1,1}}{1!} (2x)y + \dots + \frac{c_{p,1}}{p!} (2x)y^p +\dots \notag \\
&+ \frac{c_{1,2}}{1!} (2x)^2 y + \dots + \frac{c_{p,2}}{p!} (2x)^2 y^p +\dots. \label{A:1}
\end{align}
It follows from our Convention~\ref{Con2} and eqn (\ref{A:1}) that 
\begin{align}\label{A:v}
v(x,y) = \frac{1}{1-2x} + e^y -1 + \sum_{p,q\geq 1 }c_{p,q} \frac{ (2x)^q y^p}{p!}.
\end{align}
We feed this observation into our recurrence~(\ref{A:cpqrecurrence})
and use similar arguments for the right hand side of it: 
\begin{align*}
v(x,y)- \frac{1}{1-2x} - e^y +1 &= \int \sum_{p\geq 1,q \geq 0} \frac{c_{p-1,q}}{(p-1)!} (2x)^{q} y^{p-1}  dy - e^y \\
&+ 2x \bigg(\sum_{p,q \geq 0} \frac{c_{p,q}}{p!} (2x)^q y^p - \frac{1}{1-2x} \bigg) \\
&+ 2\sum_{p,q \geq 1} p\frac{c_{p-1,q-1}}{p!} (2x)^q y^p + 2 \sum_{p,q \geq 1} q\frac{c_{p-1,q-1}}{p!} (2x)^q y^p \\
&-  2 \sum_{p,q \geq 1} \frac{c_{p-1,q-1}}{p!} (2x)^q y^p \\
&= \int \sum_{p\geq 1,q \geq 0} \frac{c_{p-1,q}}{(p-1)!} (2x)^{q} y^{p-1}  dy - e^y \\ 
&+ 2x \bigg(\sum_{p,q \geq 0} \frac{c_{p,q}}{p!} (2x)^q y^p - \frac{1}{1-2x} \bigg) \\
&+ 4xy \sum_{p,q \geq 0} \frac{c_{p-1,q-1}}{(p-1)!} (2x)^{q-1} y^{q-1} \\
&+4x\int  \sum_{p,q \geq 1} \frac{qc_{p-1,q-1}}{(p-1)!} (2x)^{q-1} y^{p-1} dy\\
&-  4x \int \sum_{p,q \geq 1} \frac{c_{p-1,q-1}}{(p-1)!} (2x)^{q-1} y^{p-1} dy.
\end{align*}
Thus, we have 
\begin{align*}
v(x, y)  - \frac{1}{1-2x} - e^y +1 &= \int v(x,y) dy - e^y + 2x\ v(x,y) - \frac{2x}{1-2x} + 4xy\ v(x,y)\\
&+ x \int \bigg( \frac{\partial}{\partial x} (2x\ v(x,y))\bigg) dy - 4x \int v(x,y) dy,
\end{align*}
or equivalently,
\begin{align*}
(1-2x- 4xy) v(x, y) &= (1-4x)\int v(x,y) dy  + x \int \bigg( \frac{\partial}{\partial x} (2x v(x,y))\bigg) dy.
\end{align*} 

Now differentiating with respect to $y$ gives us a PDE: 
\begin{align*}
-4x \ v(x,y) +(1-2x-4xy) \frac{\partial v(x,y)}{\partial y} = (1-4x)\ v(x,y) 
+ x \bigg( 2v(x,y) + 2x \frac{\partial v(x,y)}{\partial x} \bigg),
\end{align*}
which we reorganize as in 
\begin{align}\label{A:PDE2}
(-2x^2) \frac{\partial v(x,y)}{\partial x} + (1-2x-4xy)\frac{\partial v(x,y)}{\partial y}=  (1+2x)\ v(x,y).
\end{align}
Here, we have the obvious initial conditions
$$
v(0,y) = e^y\ \text{ and } \ v(x,0) = \frac{1}{1-2x}.
$$

Solutions of such PDE's are easily obtained by applying the method of ``characteristic curves.''
Our characteristic curves are $x(r,s), y(r,s)$, and $v(r,s)$. Their tangents are equal to 
\begin{equation}\label{charcurve}
\frac{\partial x }{\partial s} = -2x^2, \qquad \frac{\partial y }{\partial s } 
= 1- 2x- 4xy, \qquad \frac{\partial v }{\partial s } = (1+2x)v,
\end{equation}
with the initial conditions 
$$
x(r,0) = r,\;\; y(r, 0) = 0,\;\; \text{and}\;\; v(r, 0) = \frac{1}{1-2r}.
$$
From the first equation given in (\ref{charcurve}) and its initial condition underneath, 
we have 
\begin{align}\label{A:y is simple}
x(r,s)= \frac{r}{2rs+1}.
\end{align}
Plugging this into the second equation gives us 
$\frac{\partial y }{\partial r } =  1- \frac{2}{2rs+1}(1+2y)$, which is a first order linear ODE.  
The general solution for this ODE is 
\begin{align}\label{A:x is complicated}
y(r,s)=\frac{3s + 4r^2 s^3 - 6r^2s^2 + 6rs^2-6rs }{3(2rs+1)^2}.
\end{align}
Finally, from the last equation in (\ref{charcurve}) together with its initial condition we conclude that 
$$
v(r,s) = \frac{e^{s} (2rs+1)}{1-2r}.
$$
In summary, we outlined the proof of our next result.
\begin{Theorem}\label{T:GF2}
Let $v(x,y)$ denote the power series $\sum_{p,q\geq 0} c_{p,q} (2x)^q \frac{y^p}{p!}$.
If $r$ and $s$ are the variables related to $x$ and $y$ as in equations 
(\ref{A:x is complicated}) and (\ref{A:y is simple}), then we 
\begin{align}\label{A:simple}
v(r,s) = \frac{e^{s} (2rs+1)}{1-2r}.
\end{align}
\end{Theorem}

We finish this section with a remark.
\begin{Remark}\label{R:expensive}
Although we solved our PDE by using the 
useful method of characteristic curves, 
the answer is given as a function of 
transformed coordinates $r$ and $s$. 
Actually, we can find the solution in $x$ and $y$.
Indeed, it is clear from the outset
that the general solution $\tilde{S}(x,y)$ of (\ref{A:PDE2}) is given by 
\begin{align}\label{A:PdeSolt}
\tilde{S}(x,y) = \frac{ e^{1/(2x)} F(\frac{6xy+3x-1}{6x^3})}{x},
\end{align}
where $F(z)$ is some function in one-variable.
(This can easily be verified by substituting $\tilde{S}(x,y)$ 
into the PDE.)
Let us find a concrete expression for $F(z)$ here so that 
the initial condition $\tilde{S}(x,y)=v(x,y)$ holds true. 
To this end, we set $y=0$. In this case, we know that $v(x,0) = \frac{1}{1-2x}$.
Therefore, $F(z)$ satisfies the following equation:  
\begin{align}\label{A:Gsolt}
\frac{ e^{1/(2x)} F(\frac{3x-1}{6x^3})}{x} = \frac{1}{1-2x} \;\;\; \text{or} \;\;\;
F \biggr(\frac{3x-1}{6x^3} \biggr) = \frac{xe^{-1/(2x)}}{1-2x}.
\end{align}
The inverse of the transformation $z= \frac{3x-1}{6x^3}$ which appears in (\ref{A:Gsolt}) 
is given by 
\begin{align}\label{A:backtothefuture}
x=\frac{1}{6^{1/3} (-3 z^2 + \sqrt{3} \sqrt{-2 z^3 + 3 z^4})^{1/3}} 
+ \frac{(-3 z^2 + \sqrt{3} \sqrt{-2 z^3 + 3 z^4})^{1/3}}{6^{2/3} z}.
\end{align}
By back substitution of (\ref{A:backtothefuture}) into (\ref{A:Gsolt}), 
we find an expression for $F(z)$, which in turn will be evaluated at 
$\frac{6xy+3x-1}{6x^3}$ (as in (\ref{A:PdeSolt})). Obviously the resulting 
expression is very complicated, however, this way we can write the solution
of our PDE in $x$ and $y$ only. 
\end{Remark}

\section{A combinatorial interpretation}\label{S:Interpretation}

Recall our claim (Proposition~\ref{P:anotherpaththeorem}) 
from Introduction that one can compute the 
numbers $c_{p,q}$ as a sum 
$\sum_{L\in \mc{D}(p,q)} \omega(L)$,
where $\mc{D}(p,q)$ denotes the set of all 
Delannoy paths that ends at $(p,q)$. 
(Here $\omega(L)$ is the weight of the Delannoy path
$L$, which is defined in (\ref{A:weightofL}).)
Recall also that a weighted $(p,q)$ Delannoy path is a word of the form 
$W:=K_1\dots K_r$, where $K_i$'s $(i=1,\dots, r)$ 
are labeled steps $K_i=(L_i,m_i)$ such that 
\begin{itemize}
\item $L_1\dots L_r$ is a Delannoy path from $\mc{D}(p,q)$;
\item if $L_i$ ($1\leq i \leq r$) is a $k$-th 
diagonal step, then $2\leq m_i \leq  2k-1$.
\end{itemize}
Theorem~\ref{T:anotherpaththeorem}
states that there is a bijection between the set of 
weighted $(p,q)$ Delannoy paths and the set of ssymmetric $(2p,2q)$
clans. Our goal in this section is to prove these statements.

\begin{proof}[Proof of Proposition~\ref{P:anotherpaththeorem}]

Let $c_{p,q}'$ denote the sum $\sum_{L\in \mc{D}(p,q)} \omega(L)$.
As a convention we define $c_{0,0}'=1$. 
Recall that $n$ stands for $p+q$. 
We prove our claim $c_{p,q}'=c_{p,q}$ by induction on $n$. 
Obviously, if $n=1$, then $(p,q)$ is either $(0,1)$ or $(1,0)$,
and in both of these cases, there is only one step which
either $N$ or $E$. Therefore, $c_{p,q}'=1$ in this case. 
Now, let $n$ be a positive integer and we assume 
that our claim is true for all $(p,q)$ with $(p,q)=n$. 
We will prove that $c_{p,q}=c_{p,q}'$, whenever $p+q=n+1$. 
To this end, we look at the possibilities for the ending step 
of a Delannoy path $L=L_1\dots L_r\in \mc{D}(p,q)$. 
If $L_r$ is a diagonal step, then 
$$
\omega(L) = (2(p+q)-1)\omega(L_1\dots L_{r-1}).
$$
In particular, $L_1\dots L_{r-1} \in \mc{D}(p-1,q-1)$. 
If $L_r$ is from $\{N,E\}$, then 
$$
\omega(L) = \omega(L_1\dots L_{r-1}).
$$
In particular, $L_1\dots L_{r-1} \in \mc{D}(p-1,q)$ or
$L_1\dots L_{r-1} \in \mc{D}(p,q-1)$. depending 
on $L_r=E$ or $L_r=N$. 
We conclude from these observations that 
\begin{align*}
c_{p,q}' &= c_{p-1,q}' + c_{p,q-1}'+2(p+q-1) c_{p-1,q-1}' \\
&= c_{p-1,q} + c_{p,q-1} + 2(p+q-1) c_{p-1,q-1}\qquad \text{(by induction
hypothesis)}\\
&= c_{p,q}.
\end{align*}
This finishes the proof of our claim. 
\end{proof}

The proof of Theorem~\ref{T:anotherpaththeorem} 
is based on the same idea however it requires
more attention in some of the constructions that are involved.

\begin{proof}[Proof of Theorem~\ref{T:anotherpaththeorem}]
Let $d_{p,q}$ denote the cardinality of $\mc{D}^w(p,q)$. 
We will prove that $d_{p,q}$ 
obeys the same recurrence as $c_{p,q}$'s and it satisfies the 
same initial conditions.

Let $\gamma = c_1\dots c_{2n}$ be a ssymmetric $(2p,2q)$
clan and let $\pi=\pi_\gamma$ denote the signed involution 
$$
\pi =(i_1, j_1) \dots (i_{2k},j_{2k}) l_1^{s_1} \dots l_{2n-4k}^{s_{2n-4k}},\
 \text{ where } s_1,\dots, s_{2n-4k}\in \{+,-\},
$$ 
which is given by $\pi=\widetilde{\varphi}(\gamma)$ (Here, 
$\widetilde{\varphi}$ is the map that is constructed 
in the proof of Lemma~\ref{L:clanstoinvolutions2}.)
We will construct a weighted $(p,q)$ Delannoy path $W=W_\gamma$
which is uniquely determined by $\pi$. 

First, we look at the position of $2n$ in $\pi$. 
If it appears as a fixed point with a $+$ sign, 
then we draw an $E$ step between $(p,q)$ and $(p-1,q)$. 
If it appears as a fixed point with a $-$ sign, then we draw a 
an $N$-step between $(p,q)$ and $(p,q-1)$. We label
both of these steps by 1 to turn them into labeled steps.
Next, we remove the fixed points $1$ and $2n$ from $\pi$
and then subtract 1 from each remaining entry. 
The result is a either signed $(2(p-1),2q)$ involution 
or a signed $(2p,2(q-1))$ involution. 
Now, by our induction hypothesis, in the first case, there are 
$d_{p-1,q}$ possible ways of extending this path to 
a weighted $(p,q)$ Delannoy path. In a similar manner,
in the latter case there are $d_{p,q-1}$ possible ways of extending 
it to a weighted $(p,q)$ Delannoy path. 

Now we assume that $2n$ appears in a 2-cycle in $\pi$, say $(i_s,j_s)$,
where $1\leq s \leq k$. Then $(i_r,j_r)=(i,2n)$, for some $i\in \{2,\dots, 2n-1\}$. 
Then by the symmetry condition, there is a partnering
2-cycle, which is necessarily of the form $(1,i' )$ for some $i'$.
In this case,
we draw a $D$-step between $(p,q)$ and $(p-1,q-1)$ and we 
label this step by $i$. 
Then we remove the two cycle $(i,2n)$ as well
as its partner $(1,i')$ from $\pi$.
Let us denote the resulting object by $\pi^{(1)}_0$. 
To get rid of the gaps created by the removal of two 2-cycles, 
we renormalize the remaining entries by appropriately 
subtracting numbers so that the resulting object, 
which we denote by $\pi^{(1)}$ has every number from 
$\{1,\dots, 2n-4\}$ appears in it exactly once. It is easy to 
see that we have a signed $(2(p-1),2(q-1))$ involution 
which corresponds to a ssymmetric $(2(p-1),2(q-1))$ clan 
under $\widetilde{\varphi}^{-1}$. Now, the label of 
this diagonal step can be chosen as one the 
$2(p+q-1)$ numbers from $\{2,\dots, 2n-1\}$. 
Finally, let us note that there are $d_{p-1,q-1}$ 
possible ways to extend this labeled diagaonal step 
to a weighted $(p,q)$ Delannoy path. 

Combining our observations we see that, 
starting with a random ssymmetric $(2p,2q)$ clan, 
there are exactly 
\begin{align}\label{A:total}
d_{p-1,q}+d_{p,q-1} + 2(p+q-1) d_{p-1,q-1}
\end{align}
possible weight Delannoy paths that we can 
construct. By induction hypothesis the number (\ref{A:total})
is equato $c_{p,q}$.
This finishes our proof. 

\end{proof}

Let us illustrate our construction by an example.

\begin{Example}
Let $\gamma$ denote the ssymmetric $(10,6)$ clan 
$$
\gamma = 4\ +\ 6\ -\ +\ 1\ 1\ +\ +\ 2\ 2\ +\ -\ 4\ +\ 6,
$$
and let $\pi$ denote the corresponding signed involution 
$$
\pi = (1,14)(3,16)(6,7)(10,11) 2^+ 4^- 5^+ 8^+ 9^+ 12^+ 13^-15^+.
$$ 
The steps of our constructions 
are shown in Figure~\ref{F:last pic}.
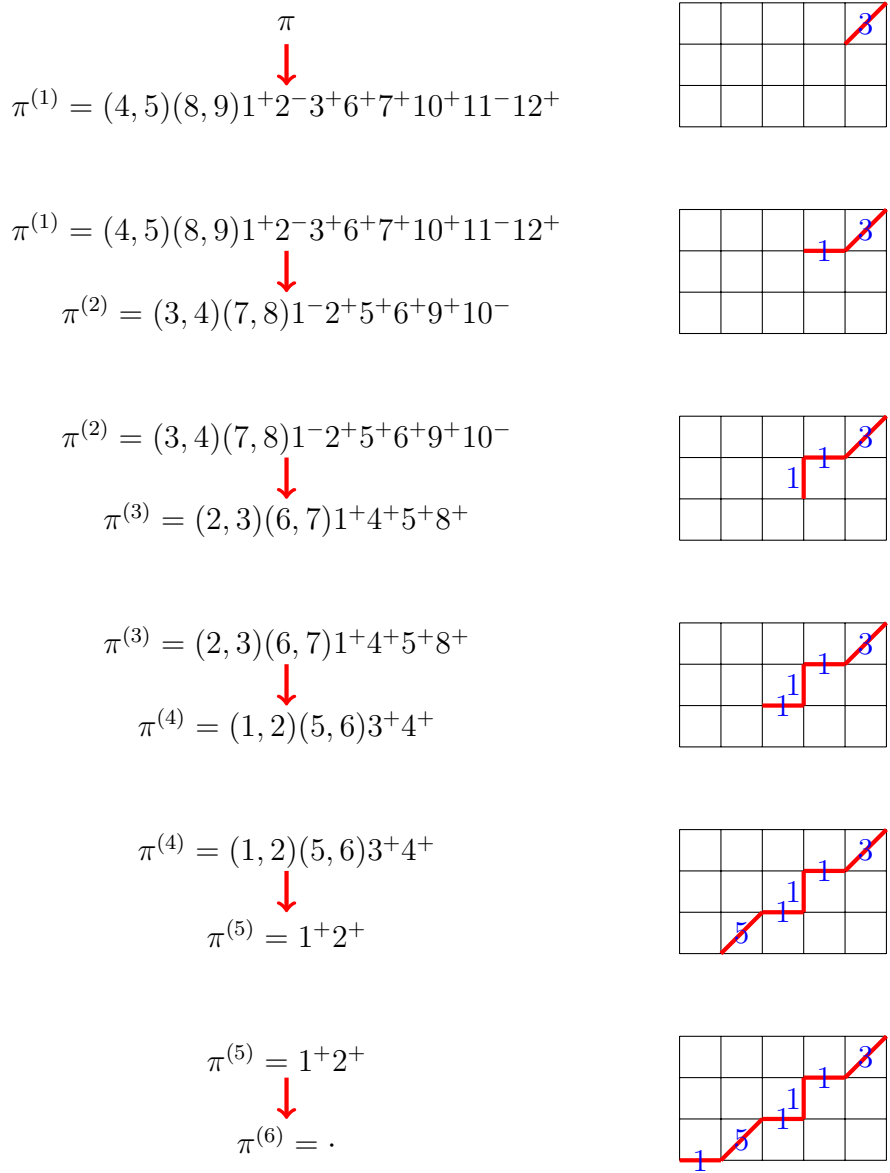
\begin{figure}[htp]
\begin{center}
\begin{tikzpicture}[scale=.55]
\begin{scope}[xshift = -5cm]
\node at (0,2.5) {$\pi$};
\node at (0,.5) {$\pi^{(1)}= (4,5)(8,9)1^+ 2^- 3^+ 6^+ 7^+ 10^+ 11^- 12^+$};
\draw[ultra thick,red, ->] (0,2) to (0,1);
\end{scope}
\begin{scope}[xshift = 4.5cm]
\draw (0,0) grid (5,3);
\draw[ultra thick,red] (5,3) to (4,2);
\node[blue] at (4.5, 2.5) {$3$};
\end{scope}
\begin{scope}[xshift = -5cm,yshift=-5cm]
\node at (0,2.5)  {$\pi^{(1)}= (4,5)(8,9)1^+ 2^- 3^+ 6^+ 7^+ 10^+ 11^- 12^+$};
\node at (0,.5) {$\pi^{(2)}= (3,4)(7,8)1^- 2^+ 5^+ 6^+ 9^+ 10^-$};
\draw[ultra thick,red, ->] (0,2) to (0,1);
\end{scope}
\begin{scope}[xshift = 4.5cm,yshift=-5cm]
\draw (0,0) grid (5,3);
\draw[ultra thick,red] (5,3) to (4,2);
\draw[ultra thick,red] (4,2) to (3,2);
\node[blue] at (4.5, 2.5) {$3$};
\node[blue] at (3.5, 2) {$1$};
\end{scope}
\begin{scope}[xshift = -5cm,yshift=-10cm]
\node at (0,2.5) {$\pi^{(2)}= (3,4)(7,8)1^- 2^+ 5^+ 6^+ 9^+ 10^-$};
\node at (0,.5) {$\pi^{(3)}= (2,3)(6,7)1^+ 4^+ 5^+ 8^+$};
\draw[ultra thick,red, ->] (0,2) to (0,1);
\end{scope}
\begin{scope}[xshift = 4.5cm,yshift=-10cm]
\draw (0,0) grid (5,3);
\draw[ultra thick,red] (5,3) to (4,2);
\draw[ultra thick,red] (4,2) to (3,2);
\draw[ultra thick,red] (3,2) to (3,1);
\node[blue] at (4.5, 2.5) {$3$};
\node[blue] at (3.5, 2) {$1$};
\node[blue] at (2.75, 1.5) {$1$};
\end{scope}
\begin{scope}[xshift = -5cm,yshift=-15cm]
\node at (0,2.5) {$\pi^{(3)}= (2,3)(6,7)1^+ 4^+ 5^+ 8^+$};
\node at (0,.5) {$\pi^{(4)}= (1,2)(5,6) 3^+ 4^+$};
\draw[ultra thick,red, ->] (0,2) to (0,1);
\end{scope}
\begin{scope}[xshift = 4.5cm,yshift=-15cm]
\draw (0,0) grid (5,3);
\draw[ultra thick,red] (5,3) to (4,2);
\draw[ultra thick,red] (4,2) to (3,2);
\draw[ultra thick,red] (3,2) to (3,1);
\draw[ultra thick,red] (3,1) to (2,1);
\node[blue] at (4.5, 2.5) {$3$};
\node[blue] at (3.5, 2) {$1$};
\node[blue] at (2.75, 1.5) {$1$};
\node[blue] at (2.5, 1) {$1$};
\end{scope}
\begin{scope}[xshift = -5cm,yshift=-20cm]
\node at (0,2.5) {$\pi^{(4)}= (1,2)(5,6) 3^+ 4^+$};
\node at (0,.5) {$\pi^{(5)}= 1^+ 2^+$};
\draw[ultra thick,red, ->] (0,2) to (0,1);
\end{scope}
\begin{scope}[xshift = 4.5cm,yshift=-20cm]
\draw (0,0) grid (5,3);
\draw[ultra thick,red] (5,3) to (4,2);
\draw[ultra thick,red] (4,2) to (3,2);
\draw[ultra thick,red] (3,2) to (3,1);
\draw[ultra thick,red] (3,1) to (2,1);
\draw[ultra thick,red] (2,1) to (1,0);
\node[blue] at (4.5, 2.5) {$3$};
\node[blue] at (3.5, 2) {$1$};
\node[blue] at (2.75, 1.5) {$1$};
\node[blue] at (2.5, 1) {$1$};
\node[blue] at (1.5, .5) {$5$};
\end{scope}
\begin{scope}[xshift = -5cm,yshift=-25cm]
\node at (0,2.5) {$\pi^{(5)}= 1^+ 2^+$};
\node at (0,.5) {$\pi^{(6)}= \cdot$};
\draw[ultra thick,red, ->] (0,2) to (0,1);
\end{scope}
\begin{scope}[xshift = 4.5cm,yshift=-25cm]
\draw (0,0) grid (5,3);
\draw[ultra thick,red] (5,3) to (4,2);
\draw[ultra thick,red] (4,2) to (3,2);
\draw[ultra thick,red] (3,2) to (3,1);
\draw[ultra thick,red] (3,1) to (2,1);
\draw[ultra thick,red] (2,1) to (1,0);
\draw[ultra thick,red] (1,0) to (0,0);
\node[blue] at (4.5, 2.5) {$3$};
\node[blue] at (3.5, 2) {$1$};
\node[blue] at (2.75, 1.5) {$1$};
\node[blue] at (2.5, 1) {$1$};
\node[blue] at (1.5, .5) {$5$};
\node[blue] at (.5, 0) {$1$};
\end{scope}
\end{tikzpicture}
\end{center}
\caption{Algorithmic construction of the bijection onto weighted Delannoy paths.}
\label{F:last pic}
\end{figure}

\end{Example}

\newpage

\section{Appendix}\label{S:Appendix}

In this appendix, as we promised in the introduction,
we outline a method for approximating the number of 
symmetric $(2p,2q+1)$ clans with $k$ pairs, $\beta_{k,p,q}$.
Recall our notation that $A_e(x) = \sum_{l=0}^{q} \beta_{2l,p,q} x^{2l}$ 
and $A_o(x) = \sum_{l=0}^{q} \beta_{2l+1,p,q} x^{2l+1}$.

First of all, by multiplying both sides of the recurrence 
relation (\ref{A:kevenfr}) by $x^{2l}$ and summing over $l$ lead us 
to the following integral/differential equation  
\begin{align*}
A_e(x)-\beta_{0,p,q} =& (q+1) \int (A_o(x) - \beta_{2q+1,p,q} x^{2q+1})dx - \frac{x}{2}(A_o(x) - \beta_{2q+1,p,q} x^{2q+1}) \\
+& 2(pq+p+q+1)\int x(A_e(x)- \beta_{2q,p,q} x^{2q})dx \\
-& (p+q+1)x^2 (A_e(x) - \beta_{2q,p,q} x^{2q}) 
+\frac{x^3}{2}A_e^{'}(x) - q\beta_{2q,p,q} x^{2q+2}.
\end{align*}
We get rid of the integrals by taking the derivative with respect to $x$
and then we reorganize our equation which is now a second order ODE as in
\begin{align*}
x^3A_e^{''}(x) - \biggr((2p+2q-1)x^2 +2\biggr)A_e^{'}(x) +4pq x A_e(x) - x A_o^{'}(x) + (2q+1)A_0(x)
 =0.
\end{align*}
By applying a similar procedure to the recurrence relation (\ref{A:koddfr}) 
and also by using the fact that $\beta_{1,p,q} = p \beta_{0,p,q}$, 
we obtain our second second order ODE: 
\begin{align*}
x^3A_o^{''}(x) - \biggr((2p+2q-1)x^2 + 2\biggr)A_o^{'}(x) +(4pq+2p-2q-1)x A_o(x)
- xA_e^{'}(x) + 2p A_e(x) =0.
\end{align*}
Note that we the following initial conditions that follow from 
the definitions of $A_e(x)$ and $A_o(x)$:
\begin{align*}
A_e(0) =&\beta_{0,p,q} \;  \text{and}\; A_o(0) = 0,  \\
A_e^{'}(0)=& 0 \; \text{and} \;  A_o^{'}(0) = p\beta_{0,p,q} = \beta_{1,p,q}.
\end{align*}
We will reduce our second order system to a first order ODE 
by setting $u(x):=A_e^{'}(x)$ and $v(x):=A_o^{'}(x)$. Then 
\begin{align*}
x^3 u'(x) =& ((2p+2q-1)x^2 +2 )u(x)-4pqx A_e(x) - (2q+1)A_0(x) + xv(x) \\
x^3 v'(x) =& ((2p+2q-1)x^2 +2) v(x)- (4pq+2p-2q-1)x A_o(x) -2p A_e(x) + x u(x) \\
x^3 A_e^{'}(x) =& x^3 u(x) \\
x^3 A_o^{'}(x) =& x^3 v(x).
\end{align*}
We write this system in matrix form 
$$
x^3 X' = A(x) X,
$$ 
where $A(x)$ the $4\times 4$ matrix as in (\ref{A:mainODE}).
Note that our initial conditions become 
\begin{align}\label{A:ODE1}
X(0) = \begin{bmatrix}
u(0)\\
v(0) \\
A(0)\\
A(0)
\end{bmatrix} =
\begin{bmatrix}
0\\
p\beta_{0,p,q} \\ 
\beta_{0,p,q}\\
0
\end{bmatrix}. 
\end{align}
Once a system of first order ordinary differential equations of this type is given, 
formal series solutions can always be obtained by carrying out the computational procedure, 
which is outlined in \cite{Tur}. 
We will use those techniques to solve the above system of first order ordinary differential equations.

Before proceeding any further let us define the matrices $A_0,A_1,\dots$ 
by decomposing the coefficient matrix $A(x)$: 
\begin{align*}
A(x) =& \sum_{k= 0}^{\infty} A_k x^k = 
\begin{bmatrix}
2 & 0 & 0 & -(2q+1) \\
0 & 2 & -2p & 0 \\
0 & 0 & 0 & 0 \\
0 & 0 & 0 & 0
\end{bmatrix} x^0
+ \begin{bmatrix}
0 & 1 & -4pq & 0 \\
1 & 0 & 0 & -(4pq+2q-2q-1) \\
0 & 0 & 0 & 0 \\
0 & 0 & 0 & 0
\end{bmatrix} x^1 \\
+& \begin{bmatrix}
(2q+2q-1) & 0 & 0 & 0 \\
0 & (2p+2q-1) & 0 & 0 \\
0 & 0 & 0 & 0 \\
0 & 0 & 0 & 0
\end{bmatrix} x^2
+ \begin{bmatrix}
0 & 0 & 0 & 0\\
0 & 0 & 0 & 0 \\
1 & 0 & 0 & 0 \\
0 & 1 & 0 & 0
\end{bmatrix} x^3 + \boldsymbol{0}x^4 + \dots.
\end{align*}
Since the eigenvalues of the leading matrix $A_0$ fall into two groups, 
namely $\lambda_1 = \lambda_2 = 0$ and $\lambda_3 = \lambda_4 = 2$, 
there exists a normalizing transformation matrix $P$ obtained from the 
Jordan canonical form of $A_0$. More precisely, since
\begin{align*}
B_0 =& \begin{bmatrix}
0 & 0 & 0 & 0 \\
0 & 0 & 0 & 0 \\
0 & 0 & 2 & 0 \\
0 & 0 & 0 & 2
\end{bmatrix}\\
 =&
\begin{bmatrix}
0 & 0 & 0 & 1 \\
0 & 0 & 1 & 0 \\
0 & 1 & -p & 0 \\ 
1 & 0 & 0 & -\frac{2q+1}{2}
\end{bmatrix}
\begin{bmatrix}
2 & 0 & 0 & -(2q+1) \\
0 & 2 & -2p & 0 \\
0 & 0 & 0 & 0 \\
0 & 0 & 0 & 0
\end{bmatrix}
\begin{bmatrix}
\frac{2q+1}{2} & 0 & 0 & 1 \\
0 & p & 1 & 0 \\
0 & 1 & 0 & 0 \\
1 & 0 & 0 & 0 \\
\end{bmatrix} \\
=& P^{-1} A_0 P.
\end{align*}
the normalizing transformation $X = PY$ turns our system into 
\begin{align}\label{A:ODE2}
x^3 Y^{'} = B(x) Y; \;\;\text{with} \;\;
Y(0) = \begin{bmatrix}
0 \\
\beta_{0,p,q} \\
0 \\
0 
\end{bmatrix}
\end{align}
where
\begin{align*}
B(x) =& P^{-1} A(x) P \\
= &
\begin{bmatrix}
0 & px^3 & x^3 & 0 \\
\frac{2q+1}{2}x^3 & 0 & 0 & x^3 \\
-\frac{2q + 1}{2} (px^3 +4px - 3x) & p(2p+2q-1)x^2 & (2p+2q-1)x^2+2 & x- px^3 \\
\frac{(2q+1)(2p+2q-1)}{2}x^2 & (p-4pq)x - \frac{p(2q+1)}{2}x^2 & x- \frac{2q+1}{2}x^3 & (2p+2q-1)x^2+2
\end{bmatrix}\\
=& \begin{bmatrix}
0 & 0 & 0 & 0\\
0 & 0 & 0 & 0\\
0 & 0 & 2 & 0\\
0 & 0 & 0 & 2
\end{bmatrix}
+ \begin{bmatrix}
0 & 0 & 0 & 0 \\
0 & 0 & 0 & 0 \\
-\frac{(2q+1)(4p-3)}{2} & 0 & 0 & 1 \\
0 & p(1- 4q) & 1 & 0
\end{bmatrix} x \\
+& \begin{bmatrix}
0 & 0 & 0 & 0 \\
0 & 0 & 0 & 0\\
0 & p(2p+2q-1) & 2p+2q-1 & 0 \\
\frac{(2q+1)(2p+2q-1)}{2} & 0 & 0 & 2p+2q-1 
\end{bmatrix} x^2\\
+& \begin{bmatrix}
0 & p & 1 & 0 \\
\frac{2q+1}{2} & 0 & 0 & 1 \\
-\frac{p(2q+1)}{2} & 0 & 0 & -p \\
0 & -\frac{p(2q+1)}{2} & -\frac{2q+1}{2} & 0
\end{bmatrix} x^3 + \boldsymbol{0}x^4 + \boldsymbol{0} x^5 + \dots 
\end{align*}
We denote the coefficient matrix of 
$x^i$ ($i=0,1,2,\dots$) in $B(x)$ by $B_i$. Thus, 
$$
B(x) = B_0 + B_1 x + B_2 x^2 +B_3 x^3.
$$

\vspace{.5cm}

We will work with a system that is obtained from $B(x)$ by 
a ``shearing'' transformation.
Let $Q$ be a formal power series of the form $Q=\sum Q_r x^r$ 
with $Q_r$'s are some constant matrices of order $4$.
We assume that our desired solution $Y=Y(x)$ 
for $x^3 Y' = B Y$ is of the form $Y=QZ$ for some $4\times 1$ column 
matrix $Z=Z(x)$. 
Formally substituting $QZ$ into $x^3 Y' = B(x) Y$ 
will give us a new ODE:
$$
x^3 (QZ)' = B QZ \Rightarrow x^3(Q'Z+QZ') = BQZ\  \text{ or } \ x^3 Z' = (Q^{-1} B Q+ x^3 Q^{-1}Q')Z.
$$
Let $C$ denote the formal power series $\sum C_r x^r$ that is defined by 
\begin{align}\label{A:Clidiff0}
Q^{-1} B Q+ x^3 Q^{-1}Q' = C = \sum C_r x^r,
\end{align}
hence our ODE is equivalent to 
\begin{align}\label{A:Clidiff}
x^3 Z' &= C Z. 
\end{align}
By multiplying both sides of (\ref{A:Clidiff0}) with $Q$ and rearranging 
we obtain a new ODE whose solution will lead to a solution of (\ref{A:realnewone}):
\begin{align}\label{A:realnewone}
x^3 Q' = Q C - B Q.
\end{align}
To solve (\ref{A:realnewone}) we simply substitute 
$B=\sum B_r x^r$, $Q=\sum Q_r x^r$ and $C=\sum C_r x^r$
and the equate coefficients. Then we get the following relations
which we call as our fundamental recurrences. 
\begin{enumerate}
\item[(i)] $0= Q_0 C_0 - B_0 Q_0$;
\item[(ii)] $0= (Q_0 C_1 - B_1 Q_0)+(Q_1C_0 - B_0Q_1)$;
\item[(iii)] $(r-2)Q_{r-2} = \sum_{i=0}^r (Q_i C_{r-i} - B_{r-i} Q_i)$  for $r\geq 2$. 
\end{enumerate}
We will recursively assign specific values to 
the matrices $Q_i$, $i=0,1,2,\dots$ which
will allow us to solve (\ref{A:realnewone}).
Along the way we will determine the 
series $C(x)=\sum x^i C_i$, which is what 
we want to solve in the first place. 
Indeed, our goal is to choose $Q_i$'s
in such a way that $C_i$'s  
become block diagonal. 
To this end, we assume that $Q_i$ ($i=0,1,2,\dots$)
is a block anti-diagonal matrix: 
\begin{align*}
Q_i = 
\begin{bmatrix}
0 & Q_i^{12} \\
Q_i^{21} & 0
\end{bmatrix} 
\end{align*}
for some $2\times 2$ matrices $Q_i^{12},Q_i^{21}$ ($i=0,1,2,\dots$). 
\vspace{.5cm}

{\em Step 1.} We choose $Q_0=I_4$,
the $4\times 4$ identity matrix. It follows from (i) 
that 
$$
C_0 = B_0
= \begin{bmatrix}
0 & 0 & 0 & 0\\
0 & 0 & 0 & 0\\
0 & 0 & 2 & 0\\
0 & 0 & 0 & 2
\end{bmatrix}.
$$
We have a remark in order.
\begin{Remark}\label{R:alwaysblockdiagonal}
Let us point out that, since 
\begin{align}\label{A:for future use}
Q_i B_0 -C_0 C_i &= 
\begin{bmatrix}
0 & 2 Q_i^{12} \\
-2Q_i^{21} & 0
\end{bmatrix} \ \text{ for } i=1,2,\dots
\end{align}
by using the fundamental recurrences (ii) and (iii)
we will always be able to choose $Q_i^{12}$ and 
$Q_i^{21}$ so that $C_i$ is of the form 
$$
C_i =
\begin{bmatrix}
C_i^{11} &0 \\
0 & C_i^{22}
\end{bmatrix},
$$
where $C_i^{11}$ and $C_i^{22}$ are $2\times2$ 
matrices.
\end{Remark}

\vspace{.5cm}

{\em Step 2.} By (ii) and Step 1, $C_1 = B_1 -Q_1 C_0+ B_0 Q_1$, so
we set 
\begin{align*}
Q_1 = \begin{bmatrix}
0 & 0 & 0 & 0 \\
0 & 0 & 0 & 0 \\
-\frac{(2q+1)(4p-1)}{4} & 0 & 0 & 0 \\
0 & \frac{p(4q-1)}{2} & 0 & 0
\end{bmatrix}
\implies
C_1 = \begin{bmatrix}
0 & 0 & 0 & 0 \\
0 & 0 & 0 & 0 \\
0 & 0 & 0 & 1 \\
0 & 0 & 1 & 0 
\end{bmatrix}.
\end{align*}

\vspace{.5cm}

{\em Step 3.} By (iii) and Steps 1,2, $C_2 = B_2 -Q_1 C_1+ B_1 Q_1 -Q_2C_0 + B_0 Q_2$, 
so we set 
\begin{align*}
Q_2 = \begin{bmatrix}
0 & 0 & 0 & 0 \\
0 & 0 & 0 & 0 \\
0 & -\frac{p(4p+8q-3)}{4} & 0 & 0 \\
-\frac{(2q+1)(4q-1)}{8} & 0 & 0 & 0
\end{bmatrix}
\implies
C_2 = \begin{bmatrix}
0 & 0 & 0 & 0 \\
0 & 0 & 0 & 0 \\
0 & 0 & 2p+2q-1 & 0 \\
0 & 0 & 0 & 2p+2q-1 
\end{bmatrix}.
\end{align*}
In a similar manner, we put 
\begin{align*}
Q_3 = \begin{bmatrix}
0 & 0 & \frac{1}{2}  & 0\\
0 & 0 & 0 & \frac{1}{2}  \\
\frac{(2q+1)(16p^2+16pq-1)}{16} & 0 & 0 & 0 \\
0 & \frac{-p(16pq+16q^2-8p-16q+1)}{8} & 0 & 0
\end{bmatrix}
\implies
C_3 = \begin{bmatrix}
0 & p & 0 & 0 \\
\frac{2q+1}{2} & 0 & 0 & 0 \\
0 & 0 & 0 & -p \\
0 & 0 & -\frac{2q+1}{2} & 0 
\end{bmatrix}.
\end{align*}

The above computations are in some
sense are our initial conditions. 
To get a better understanding of the general case we make 
a few more preliminary observations and formal computations.
\begin{align}\label{A:1 of many}
C_i^{jj} = B_i^{jj} \ \text{ for } i =0,1,2,3 \ \text{ and } \ j=1,2.
\end{align}
\begin{align}
Q_iC_1 = 
\begin{bmatrix}
0 & Q_i^{12} C_1^{22} \\
0 & 0 
\end{bmatrix} \hspace{.5cm} \text{ and }  \hspace{.5cm}
B_1Q_i = 
\begin{bmatrix}
0 & 0 \\
B_1^{22}Q_i^{21} & B_1^{21}Q_i^{12} 
\end{bmatrix} 
\end{align}

\begin{align}\label{A:2 of many}
Q_iC_2 = 
\begin{bmatrix}
0 & Q_i^{12} C_2^{22} \\
0 & 0 
\end{bmatrix} \hspace{.5cm} \text{ and }  \hspace{.5cm}
B_2Q_i = 
\begin{bmatrix}
0 & 0 \\
B_2^{22}Q_i^{21} & B_2^{21}Q_i^{12} 
\end{bmatrix} 
\end{align}

\begin{align}\label{A:3 of many}
Q_iC_3 = 
\begin{bmatrix}
0 & Q_i^{12} C_3^{22} \\
Q_i^{21} C_3^{11} & 0 
\end{bmatrix} \hspace{.5cm} \text{ and }  \hspace{.5cm}
B_3Q_i = 
\begin{bmatrix}
B_3^{12}Q_i^{21} & B_3^{11} Q_i^{12} \\
B_3^{22}Q_i^{21} & B_3^{21}Q_i^{12} 
\end{bmatrix} 
\end{align}

\begin{align}\label{A:4 of many}
Q_iC_j = 
\begin{bmatrix}
0 & Q_i^{12} C_j^{22} \\
Q_i^{21} C_j^{11} & 0 
\end{bmatrix}
\end{align}

Finally, since $B_r = 0$, the fundamental recurrence (iii)
simplifies to 
\begin{align}\label{A:friii}
C_r = (r-2)Q_{r-2}  
-\left(\sum_{i=0}^3 (Q_{r-i} C_i - B_i Q_{r-i})\right) - \left(\sum_{i=4}^{r-1} Q_{r-i} C_i  \right).
\end{align}

\vspace{1cm}

Recall that we started with the system $x^3 X' = A(x) X$ 
which is transformed into $x^3 Y' = B(x) Y$ by conjugating 
with a constant matrix, and the latter system is transformed 
into $x^3 Z' = C(x) Z$ by the shearing transformation $Y= Q(x) Z$. 

\begin{Proposition}\label{L:first simplification}
Let $C(x) = \sum_r C_r x^r$ and $Q(x) = \sum_r Q_r x^r$ 
be as in the previous paragraph. If $r\geq 4$, then 
we have 
$$
C_r =
\begin{bmatrix}
Q_{r-3}^{21} & 0 \\
0 & B_1^{21} Q_{r-1}^{12} + B_2^{21} Q_{r-2}^{12} + B_3^{21} Q_{r-3}^{12}
\end{bmatrix}.
$$
In particular, the system $x^3 Z' = C(x) Z$ decomposes
into two $2\times 2$ systems of ODE's 
\begin{align}
x^3 K' &= R(x) K   \label{A:CtoR}\\
x^3 L' &= T(x) L  \label{A:CtoT}
\end{align}
where 
\begin{align*}
R(x) &= C_3^{11} x^3 + \sum_{r\geq 4} Q_{r-3}^{21} x^r; \\
T(x) &= \sum_{i=0}^3 C_i^{22}x^3 
+  \sum_{r\geq 4} (B_1^{21} Q_{r-1}^{12} + B_2^{21} Q_{r-2}^{12} + B_3^{21} Q_{r-3}^{12}) x^r.
\end{align*}
\end{Proposition}
\begin{proof}

Since $C_r$ is a block diagonal matrix, 
recurrence (\ref{A:friii}) combined with 
equations (\ref{A:1 of many})--(\ref{A:4 of many}) gives
us the desired result.
\end{proof}

What remains is to solving the systems (\ref{A:CtoR}) and (\ref{A:CtoT}).
The former ODE is relatively easy since it does not have a singularity anymore. 
However, the second ODE (\ref{A:CtoT}) does have a singularity. 
Moreover, we still do not know the exact forms of neither $Q^{12}(x)$ 
nor $Q^{21}(x)$. On the positive side, by taking advantage of the particular 
structure of $B(x)$'s we are able to find recurrences for $R(x)$ and $T(x)$. 
\vspace{.5cm}

To find a recurrence
for the blocks of $Q_r$'s, 
we use Proposition~\ref{L:first simplification}
as well as the simplified fundamental recurrence (\ref{A:friii}) as follows:
\begin{align*}
C_r &= \begin{bmatrix}
Q_{r-3}^{21} & 0 \\
0 & B_1^{21} Q_{r-1}^{12} + B_2^{21} Q_{r-2}^{12} + B_3^{21} Q_{r-3}^{12}
\end{bmatrix}\\ =& 
\begin{bmatrix}
0 & (r-2)Q_{r-2}^{12}  \\
(r-2) Q_{r-2}^{21} & 0
\end{bmatrix} -
\begin{bmatrix}
0 & 2Q_{r}^{12}  \\
-2Q_{r}^{21} & 0
\end{bmatrix} -
\begin{bmatrix}
0 & Q_{r-1}^{12}B_{1}^{22}  \\
-B_{1}^{22}Q_{r-1}^{21} & -B_{1}^{21}Q_{r-1}^{12}
\end{bmatrix}\\ 
-&
\begin{bmatrix}
0 & Q_{r-2}^{12}B_{2}^{22}  \\
-B_{2}^{22}Q_{r-2}^{21} & -B_{2}^{21}Q_{r-2}^{12}
\end{bmatrix}-
\begin{bmatrix}
-B_{3}^{12}Q_{r-3}^{21} & Q_{r-3}^{12}B_{3}^{22} - B_{3}^{11}Q_{r-3}^{12}  \\
Q_{r-3}^{21}B_{3}^{11} -B_{3}^{22} Q_{r-3}^{21}& -B_{3}^{21}Q_{r-3}^{12}
\end{bmatrix}\\
-& \sum_{i=4}
\begin{bmatrix}
0 & Q_{r-i}^{12}  \\
Q_{r-i}^{21} & 0
\end{bmatrix} 
\begin{bmatrix}
Q_{i-3}^{21} & 0  \\
0 & B_1^{21} Q_{i-1}^{12} + B_2^{21} Q_{i-2}^{12} + B_3^{21} Q_{i-3}^{12}
\end{bmatrix} \\
=&
\begin{bmatrix}
Q_{i-3}^{21} &
\begin{matrix}
(r-2)Q_{r-2}^{12}-2Q_{r}^{12}-Q_{r-1}^{12}B_1^{22} \\
\hfill{}-Q_{r-2}^{12}B_2^{22}-Q_{r-3}^{12}B_3^{22}+B_3^{11}Q_{r-3}^{12}\\
\hfill{}- \sum_{i=4} Q_{r-i}^{12}(B_1^{21} Q_{i-1}^{12} + B_2^{21} Q_{i-2}^{12} + B_3^{21} Q_{i-3}^{12})
\end{matrix}   \\
\begin{matrix}
(r-2)Q_{r-2}^{21}+2Q_{r}^{21}+B_1^{22}Q_{r-1}^{21}+B_2^{22}Q_{r-2}^{21}\\
\hfill{}-Q_{r-3}^{21}B_3^{11}+B_3^{22}Q_{r-3}^{21}- \sum_{i=4} Q_{r-i}^{21}Q_{i-3}^{21}
\end{matrix}
 & B_1^{21} Q_{i-3}^{12} + B_2^{21} Q_{i-2}^{12} + B_3^{21} Q_{i-3}^{12}
\end{bmatrix}.
\end{align*}
Observe that the diagonal blocks do not 
give us any new information, however,
the anti-diagonal blocks do. 
By the equality of the bottom left blocks,  
we have 
\begin{align}\label{A:lowerentry}
2Q_{r}^{21}=-(r-2)Q_{r-2}^{21}
-B_1^{22}Q_{r-1}^{21}-B_2^{22}Q_{r-2}^{21}+Q_{r-3}^{21}B_3^{11}-B_3^{22}Q_{r-3}^{21}
+ \sum_{i=4}^{r-1} Q_{r-i}^{21}Q_{i-3}^{21}
\end{align}
Similarly, the equality of the top right blocks give 
\begin{align}
\begin{split}\label{A:upperentry}
2Q_r^{12} ={}& (r-2)Q_{r-2}^{12}-Q_{r-1}^{12}B_1^{22}-Q_{r-2}^{12}B_2^{22}-Q_{r-3}^{12}B_3^{22}+B_3^{11}Q_{r-3}^{12}\\
-& \sum_{i=4}^{r-1} Q_{r-i}^{12}(B_1^{21} Q_{i-1}^{12} + B_2^{21} Q_{i-2}^{12} + B_3^{21} Q_{i-3}^{12}).
\end{split}
\end{align}

Obviously, these recurrences enable us to 
write the precise forms of the ODE's (\ref{A:CtoR})
and (\ref{A:CtoT}). Both of these ODE's can now 
be solved by applying suitable shearing transformations
leading to a solution of our original equation $x^3 X'=A(x) X$.
However, due to its high computational 
cost the result is still not better than the 
expressions for $\beta_{k,p,q}$'s that
we recorded in Theorem~\ref{T:first}.

\vspace{1cm}
\textbf{Acknowledgements.}
We are thankful to Roger Howe who shared with us his manuscript on 
$K\backslash G / B$-decomposition of classical groups.

\end{document}